\documentclass[12pt]{article}
\makeatletter
\let\@fnsymbol\@arabic
\makeatother
    \usepackage[export]{adjustbox}
	\usepackage[english]{babel}
	\usepackage{booktabs}	
	\usepackage{tikz}
\usepackage{pdfpages}
	\usepackage{eqnarray}
	\usepackage{stackrel}
	\usepackage{bbm}
	\usepackage[T1]{fontenc}
\usepackage{lmodern}
\usepackage{chngcntr}
	\usepackage{enumerate}
	\usepackage{float} 		
	\usepackage[round]{natbib}
    \usepackage{epstopdf}
    \usepackage[toc,page]{appendix} 
	\usepackage[intlimits]{amsmath}
    \usepackage{amsfonts}
	\usepackage{rotating} 	
	\usepackage{graphicx} 	
\usepackage{ae,aecompl}
\usepackage{subcaption}
\usepackage{amsmath,amsthm,amssymb,enumerate,mathrsfs}
    \usepackage{epsfig}
\usepackage[pdftex]{hyperref}
	\usepackage{caption}
    \usepackage{multirow}
    \usepackage{epsfig}    
    \usepackage{lscape}    
    \usepackage{caption}
    \usepackage{subcaption}
    \usepackage{setspace}
    \usepackage{url}
		\usepackage[a4paper]{geometry}
 \theoremstyle{plain}
\newtheorem{ro}{Corollary}

\newtheorem{thm}{Theorem}

\theoremstyle{definition}

\newtheorem{rem}{Remark}
\newtheorem{lm}{Lemma}
\newtheorem{exm}{Example}
\theoremstyle{remark}
\usepackage{setspace}

\usepackage{setspace}

\usepackage{titling}

\begin{titlepage}

\title
{Linear type conditional specifications for multivariate count variables}

\author{
Yang Lu\thanks{Department of Mathematics and Statistics, Concordia University, Montreal, Canada. Email: yang.lu@concordia.ca}
\and
Wei Sun\thanks{Department of Mathematics and Statistics, Concordia University, Montreal, Canada. Email: wei.sun@concordia.ca}  }
 \date{}
\end{titlepage}
\begin{document}
\maketitle


%

\textbf{Abstract}:  This paper investigates conditional specifications for multivariate count variables.  Recently,  the spatial count data literature has proposed several conditional models such that the conditional expectations are linear in the conditioning variables.  These models are much easier to estimate than existing spatial count models based on Gaussian random field. However,  whether or not such conditional specifications are compatible have not been addressed.  We investigate two large families of conditional models, that are the compound autoregressive model and the random coefficient integer autoregressive model.  We characterize all the solutions to these two families of models at arbitrary dimensions, and find that only a handful of them admit non-trivial solutions.  We then show that if we focus on the linearity condition of the conditional expectations only,  a considerable larger family of solutions can be obtained.   This suggests that for spatial count data modeling,  semi-parametric type specifications that impose the conditional expectation structure is preferable.
  \vspace{1em}

 \vspace{1em}



\section{Introduction}
In this paper,  we investigate conditional specifications of multivariate count variables.  In general,  the distribution of a random vector $(X_1,...,X_n)$ is characterized by the conditional distributions of $X_j \mid X_1,...,X_{j-1},X_{j+1},...,X_n, j=1,...,n$ [see \cite{nerlove1976multivariate} for the discrete case and \cite{gourieroux1979characterization} for the continuous case].  This result is the cornerstone of the computational statistics literature on Gibbs sampling.  It also provides an alternative way of specifying the joint distribution through these conditional distributions,  which is sometimes easier to visualize,  especially in high dimensions [see  e.g.  \cite{heckerman2000dependency,  allen2013local, bengio2014deep, hadiji2015poisson} for recent attempts in the machine learning literature].   This approach through the conditionals, however, needs to be compatible, also called consistent, that is, the $n$ conditional distributions should also ensure the existence of a valid $n-$variate joint distribution.  This problem has been considered by, among others, \cite{arnold1989compatible, arnold2001conditionally, ip2009canonical,berti2014compatibility,  arnold2022conditional},  but most of these papers focus on continuously valued variables or finite discrete variables.  To our knowledge, very few conditional models for multivariate count data (non-negative integer valued, with no upper bound) have been considered, besides the auto-Poisson model of \cite{besag1974spatial}, also called  Obrechkoff's Poisson model \citep{arnold2001conditionally}.  In the simplest bivariate case,  this distribution postulates the conditional distributions:
$$
\ell(x \mid y) \propto  \lambda_1^x (\lambda_3^y)^x/x!,  \qquad \ell(y \mid x) \propto  \lambda_2^x (\lambda_3^x)^y/y!,\ \ \ \  \forall x, y \in \mathbb{N}_0,
$$
where $\mathbb{N}_0$ is the set of non-negative integers,  parameters $\lambda_1, \lambda_2>0$ and $\lambda_3 \in (0, 1)$.   This model has at least two downsides.  First,  the correlation between $X$ and $Y$ is negative.  As \cite{besag1974spatial} acknowledges, ``\textit{this restriction is severe and necessarily implies a `competitive' rather than `co-operative' interaction between auto-Poisson variates}". Second,  the marginal expectations of $X$ and $Y$ are not tractable,  and become even more complicated when the model is extended into higher dimensions.

In this paper, we fill the gap by studying alternative conditional specifications for multivariate count distributions, such that the conditional expectations $\mathbb{E}[X_j \mid X_1,...,X_{j-1},X_{j+1},...,X_n]$, $j=1,...,n$, are all linear in the conditioning variables.  A starting,  toy example is the following conditional linear Poisson model:
  \begin{equation}
  \label{linearpoisson}
  X \mid Y  \sim Pois( c Y+ d), \qquad   Y \mid X  \sim Pois( aX+ b).
  \end{equation}
The motivation of considering such linear type specifications is that first, linear specification is very convenient,  especially in a high-dimensional setting, and they have had many success stories in several related literatures.  For instance,  they have become the driving force in the spatial econometrics literature for continuously valued data \citep{lee2023spatial},  under the name of Spatial AutoregRessive (SAR) models.  Similarly, if in \eqref{linearpoisson},  only one conditional distribution is specified,  then the conditional Poisson model with parameters linear in the conditioning variable is called INGARCH and is one of the most popular count time series models [see e. g. \cite{ferland2006integer} and also \cite{davis2021count} for  reviews].  Second,  models like \eqref{linearpoisson} have recently appeared in the spatial count data literature.  In particular,  in the simplest bivariate case,  the models of \cite{glaser2022spatial, jung2022modelling} coincide with model \eqref{linearpoisson}.
Alternatively,  \cite{ghodsi2012first, brannas2014adaptations, chutoo2021unilateral, karlis2024multilateral, cote2025tree} suggest the following integer-autoregressive specification:
\begin{equation}
  \label{inar}
 X \mid Y \sim Bin( \alpha_1, Y) \circ   F_1 \qquad   Y \mid X  \sim  Bin( \alpha_2, X) \circ   F_2,
  \end{equation}
  where $F_1,  F_2$ are count distributions,  and $\alpha_{1},  \alpha_2$ are probability parameters.
We remark that both model \eqref{linearpoisson} and model \eqref{inar} can be written as:
\begin{equation}
  \label{inar2}
  X \mid  Y=_{(d)} \sum_{i=1}^Y Z_i +\epsilon,  \qquad Y\mid X =_{(d)} \sum_{i=1}^X W_i +\eta,
  \end{equation}
  where $=_{(d)}$ means equality in distribution,  $(Z_i)$ and $(W_i)$ are two sequences of i.i.d.  random variables, while $\epsilon$ and $\eta$ are count variables independent of $(Z_i)$ and $(W_i)$, respectively.  In model \eqref{linearpoisson},  $(Z_i), (W_i), \epsilon$ and $\eta$ are Poisson, whereas in model \eqref{inar},  $(Z_i)$ and $(W_i)$ are Bernoulli and $\epsilon, \eta$ are Poisson.  Equivalently, model \eqref{inar2} can be rewritten using the conditional pgf:
\begin{equation}
	\label{conditionalpgf}
	\mathbb{E}[u^X|Y]= \Big(\mathbb{E}[u^{Z_1}]\Big)^Y \mathbb{E}[u^\epsilon], \qquad 	\mathbb{E}[v^Y|X]= \Big(\mathbb{E}[v^{W_1}]\Big)^X \mathbb{E}[v^\eta].
\end{equation}
We call model \eqref{inar2} compound autoregressive, or affine \citep{darolles2006structural,  lu2019predictive, darolles2019bivariate}.

  We also consider an alternative model,  in which we assume $(Z_i)$ and $(W_i)$ to be Bernoulli, but we relax the i.i.d. assumption, and assume that they are only i.i.d. given stochastic probability parameters $p$ and $q$, respectively.  We call this specification random coefficient (integer-autoregressive) model. Both the compound autoregressive model and the random coefficient model imply that:
\begin{equation}
\label{linearconditionalexp}
  \mathbb{E}[X \mid Y] = c Y+d,  \qquad    \mathbb{E}[Y \mid X] = a X+b.
\end{equation}

 We say that a joint distribution is a solution to a conditional model   if it leads to the postulated conditional distributions.   In this paper, we start by completely characterizing solutions to the compound autoregressive and the random coefficient models, as well as their respective $n-$variate extensions.   We find that:
 
 \noindent $i)$ the compound autoregressive and the random coefficient models allow for two and one families of solutions,  respectively.  
 
 \noindent $ii)$ when both models are extended to the $n-$variate case, only the compound autoregressive model allows for one solution involving Poisson-gamma conjugacy.  That is,  $X_1,...,X_n$ are conditionally i.i.d.  given a latent variable $U$,  with Poisson distributions $Pois(\lambda_i U), i=1,...,n$,  while $U$ is gamma distributed. 
 
 \noindent $iii)$ given this lack of flexibility,  we consider the more general linear conditional expectation model \eqref{linearconditionalexp},  which extends both the compound autoregressive and the random coefficient models.  We also extend this model to the $n-$dimensional case and study its solutions.  Since the conditional expectations do not completely characterize the distribution,  we cannot work out all the solutions to this latter specification.  Nevertheless, we show that there are many more models that satisfy the linear conditional expectation assumption.  These results suggest that when working with spatial count data,  specifications like \eqref{linearpoisson} and \eqref{inar} might not be valid, and instead, semi-parametric specifications that focus on the conditional expectations are preferable.

  The rest of the paper is organized as follows.  Section 2 studies some elementary models, for which the conditional pmf's are easy to express.  Section 3 characterizes the solutions to the compound autoregressive model.  Section 4 characterizes the solutions of the random coefficient model. Section 5 studies the linear conditional expectation model.  Section 6 concludes.  Lengthy proofs are relegated to the Appendix.

  \section{Preliminaries: solutions of models \eqref{linearpoisson} and \eqref{inar}}

In this section, we start by checking the compatibility of the toy models \eqref{linearpoisson} and \eqref{inar}.  Our starting point is the following elementary result.

\begin{thm}[\cite{arnold2001conditionally}, Theorem 4.1]
The two conditional distributions $\ell(x \mid y)$ and $\ell(y \mid x)$ are compatible if and only if 

\noindent $i)$ the domains of $\ell(x \mid y)$ and $\ell(y \mid x)$ are identical and

\noindent $ii)$ there exist suitable functions $u(\cdot)$ and $v(\cdot)$ such that:
\begin{equation}
	\label{casen2}
\frac{\ell(x \mid y)}{\ell(y \mid x)} =u(x) v(y), \qquad  \forall x, y.
\end{equation}
\end{thm}
In other words, the ratio of the two conditional distributions should take a separable form.  For both models \eqref{linearpoisson} and \eqref{inar},  the conditional distributions have simple expression and thus we can easily use Theorem 1 to check their compatibility.
\begin{ro}
Eq.\eqref{linearpoisson} is not compatible except when $a=c=0$.
\end{ro}
This corollary implies that the models considered by \cite{glaser2022spatial}, \cite{jung2022modelling} are not compatible.
\begin{proof}
	Under \eqref{linearpoisson}, we have that:
	$$
\frac{\ell(x \mid y)}{\ell(y \mid x)} = \frac{e^{-cy-d}(cy+d)^x y!}{x! e^{-ax-b}(ax+b)^y }.
	$$
	Then, the term $\frac{(cy+d)^x }{ (ax+b)^y }$ is separable only when $a=c=0$.
\end{proof}
\begin{ro}\label{Feb22a}
Suppose that $(Z_i)$ and $(W_i)$ are Bernoulli random variables with parameters $\alpha$ and $\beta$, respectively. Then, eq.\eqref{inar2} is compatible if and only if the distributions of $\epsilon$ and $\eta$ are Poisson and their Poisson parameters $\lambda_\epsilon$ and $\lambda_\eta$ satisfy
\begin{equation}\label{Feb22c}
\frac{\alpha }{\lambda_\epsilon(1-\alpha)}=\frac{\beta }{\lambda_\eta(1-\beta)}.
\end{equation}
\end{ro}

\begin{proof}
	We have:
	$$
	\ell(x|y)=\sum_{i=0}^{\min(x,y)}\binom{y}{i}\alpha^i(1-\alpha)^{y-i} g(x-i), \qquad 	\ell(y|x)=\sum_{j=0}^{\min(x,y)}\binom{x}{j}\beta^j(1-\beta)^{x-j} h(y-j),
	$$
	where $g$ and $h$ are the pmf's of $\epsilon$ and $\eta$, respectively.
	
	If $\ell(x|y)$ and $	\ell(y|x)$ are compatible with each other, then \eqref{casen2} is satisfied.  	By taking $y=0$, we get $u(x)v(0)=\frac{g(x)}{(1-\beta)^x h(0)}$, and by taking $x=0$, we get $u(0)v(y)=\frac{ (1-\alpha)^y g(0)}{h(y)}$.  Thus, we have:
	\begin{equation}
		\label{ratioequal}
		\frac{\sum_{i=0}^{\min(x,y)}\binom{y}{i}\alpha^i(1-\alpha)^{-i} g(x-i)}{g(x)}=C\cdot\frac{\sum_{j=0}^{\min(x,y)}\binom{x}{j}\beta^j(1-\beta)^{-j} h(y-j)}{h(y)},
	\end{equation}
where $C$ is a constant.   By taking $x=y=0$, we get $C=1$.

Further, by taking $x=1$, we deduce that $\frac{h(y)}{h(y-1)}= C_1/y$, where $C_1$ is a constant. Thus, $h$ is the pmf of a Poisson distribution.  Similarly, $g$ is a Poisson pmf. Denote by $\lambda_\epsilon$ and $\lambda_\eta$ the Poisson parameters of $g$ and $h$, respectively.  Then, eq.\eqref{ratioequal} becomes:
	\begin{equation*}
	{\sum_{i=0}^{\min(x,y)}\frac{x! y!}{(x-i)! (y-i)! i!}\alpha^i(1-\alpha)^{-i} }\lambda_\epsilon^{-i}=\sum_{j=0}^{\min(x,y)}\frac{x! y!}{(x-j)! (y-j)! j!}\beta^j(1-\beta)^{-j}\lambda_\eta^{-j}.
\end{equation*}
This equality holds for any $x$ and $y$ if and only if (\ref{Feb22c}) holds.
\end{proof}
\begin{rem}
The distribution of $(X,Y)$ that solves \eqref{inar2} in Corollary \ref{Feb22a} has the stochastic representation (also called trivariate Poisson reduction):
$$
\begin{bmatrix}
X \\
Y
\end{bmatrix}
 =
\begin{bmatrix}
 Z_0+Z_1 \\
 Z_0+Z_2
\end{bmatrix},
$$
where $Z_0, Z_1, Z_2$ are independent, Poisson distributed with parameters $\lambda_0$, $\lambda_1$ and  $\lambda_2$, respectively. By the well known fact that $Z_0\mid Z_1+Z_0$ (resp., $Z_0\mid Z_2+Z_0$) is binomial distributed with number of trial parameter $Z_1+Z_0$ (resp., $Z_2+Z_0$) and probability parameter $\frac{\lambda_0}{\lambda_0+\lambda_1}$
(resp., $\frac{\lambda_0}{\lambda_0+\lambda_2}$).
Then, the conditional distributions of $X\mid Y$ and $Y\mid X$ are of the form \eqref{inar2} and $\alpha$ and $\beta$ are related to the Poisson parameters given by:
\begin{equation}\label{Feb22b}
\alpha=\frac{\lambda_0}{\lambda_0+\lambda_2},\qquad \beta=\frac{\lambda_0}{\lambda_0+\lambda_1}.
\end{equation}
Note that (\ref{Feb22b}) is compatible with (\ref{Feb22c}).

\end{rem}


We end this section by an important model that is the solution to the compound autoregressive model \eqref{inar2}.


\begin{exm}[Poisson-gamma conjugacy]
Let $U$ be gamma distributed with shape parameter $\alpha$ and rate parameter $\beta$, that is, its pdf is $\pi(u)\propto u^{\alpha-1}e^{-\beta  u}$. Assume that, given $U$, $X$ and $Y$ are conditionally independent and Poisson distributed with parameters $\lambda_1 U$ and $\lambda_2 U$, respectively. Then, we can show that $X \mid Y$ and $Y \mid X$ are both conditional negative binomial (henceforth NB).  Indeed, by the Bayes' formula, the posterior distribution of $U \mid X$ has the density:
$
\pi(u|x) =\frac{\ell(x|u) \pi(u)}{\ell(x)} \propto  e^{-\lambda_1 u} u^x  u^{\alpha-1}e^{-\beta  u},
$
which is gamma with shape parameter $\alpha+x$ and rate parameter $\lambda_1+\beta$. Then, the conditional distribution of $Y\mid X$ is has pmf:
\begin{equation*}
\int_0^\infty \Big[ e^{-\lambda_2 u } \frac{(\lambda_2 u )^y}{y!} \Big] \Big[ (\lambda_1+\beta)^{\alpha+x}\frac{u^{\alpha+x-1}e^{-(\lambda_1+\beta) u}  }{\Gamma(\alpha+x)}\Big] \mathrm{d} u= \frac{\Gamma(y+x+\alpha)}{y! \Gamma(\alpha+x)} \frac{(\lambda_1+\beta)^{\alpha+x}}{(\lambda_1+\beta+\lambda_2)^{\alpha+x+y}}\lambda_2^y,
	\end{equation*}
which is the pmf of the NB distribution with probability parameter $p_1=\frac{\lambda_1+\beta}{\lambda_1+\beta+\lambda_2}$ and number of successes parameter $r_1=\alpha+x$. Similarly, the conditional distribution of $X\mid Y$ is NB distributed with probability parameter $p_2=\frac{\lambda_2+\beta}{\lambda_1+\beta+\lambda_2}$ and number of successes parameter $r_2=\alpha+y$. As a consequence, the joint distribution $(X, Y)$ solves the compound autoregressive model \eqref{inar2}  with $Z_i$ and $W_i$ geometric distributed with probability parameters $p_2$ and $p_1$, respectively,  and $\epsilon$ and $\eta$ NB distributed with probability parameters $p_2$ and $p_1$,  and common number of successes parameter $\alpha$.
\end{exm}
We can also check that the conditional expectations are linear in the conditioning variables:
\begin{equation}
	\label{poissnogamma}
\mathbb{E}[Y|X]= \mathbb{E}[\mathbb{E}[Y|U] | X ]= \mathbb{E}[\lambda_2 U|X] =\lambda_2 \cdot\frac{\alpha+X}{\lambda_1+\beta},\ \ \ \ \ \ \mathbb{E}[X|Y]= \lambda_1\cdot \frac{\alpha+Y}{\lambda_2+\beta}.
\end{equation}
This linearity of the conditional expectation is due to the fact that $i)$ the Poisson distribution belongs to the linear exponential family; $ii)$ the distribution of $U$, which is gamma, is conjugate prior to the Poisson pmf.  It is well known that for the exponential family,  linear conditional expectation characterizes linear posterior expectation [see \cite{diaconis1979conjugate} for the proof in the continuous case and \cite{chou2001characterization} for the discrete case].




  \section{Characterization of solutions to the compound autoregressive model}
  In the last section,  we have seen that the toy models \eqref{linearpoisson} and \eqref{inar} are quite restrictive.  This leads us to go a step further  and investigate the whole family of compound autoregressive models \eqref{inar2},  as well as its higher-dimensional extensions.
  \subsection{The bivariate case}
Let us assume that (\ref{inar2}) holds, where  $(Z_i)$ and $(W_i)$ are i.i.d. sequences but not necessarily Bernoulli distributed,  and $\epsilon$, $\eta$ are count variables. The next Theorem characterizes all the solutions to this model.


\begin{thm}
The joint distribution of $(X, Y)$ satisfies model \eqref{inar2}, where $Z_i$, $W_i$, $\epsilon, \eta, X, Y$ have finite moments of all integer orders, and the support of $(X, Y)$ is the entire space $\mathbb{N}_0^2$,  if and only if one of the following cases is satisfied:
\vspace{0.5em}
 
\noindent \textbf{Case} $i)$ (The model of Corollary 2).  $(Z_i)$, $(W_i)$ are Bernoulli with probability parameters $\alpha, \beta$, respectively,  and $ \epsilon,\eta$ are Poisson with parameters $ \lambda_\epsilon,\lambda_\eta$, satisfying eq.\eqref{Feb22c}.  In this case,  both $X, Y$ are also marginally Poisson, with parameters given by
 $
\lambda_X= \frac{\alpha \lambda_\eta+ \lambda_\epsilon}{1-\alpha \beta}$ and $\lambda_Y= \frac{\beta \lambda_\epsilon+ \lambda_\eta}{1-\alpha \beta}.
 $
In this case,  the conditional expectations are:
   $$
   \mathbb{E}[X \mid Y]=\alpha Y +\lambda_\epsilon,  \qquad  \mathbb{E}[Y \mid X]= \beta X+\lambda_\eta.
   $$
\noindent \textbf{Case} $ii)$ (Extension of Example 1) The pgf's of $(Z_{i})$,  $\epsilon$,  $(W_{i})$ and $\eta$  given by:
\begin{align}
\label{zeroinflated}
&\mathbb{E}[u^{Z_{i}}]= \frac{1+\theta_2(1-u)}{1+\theta_1(1-u)}, \qquad \mathbb{E}[u^{\epsilon}]= \frac{1}{(1+\theta_1(1-u))^\delta}, \\
\label{zeroinflated2}
&\mathbb{E}[v^{W_{i}}]= \frac{1+\theta_4(1-v)}{1+\theta_3(1-v)}, \qquad \mathbb{E}[v^{\eta}]= \frac{1}{(1+\theta_3(1-v))^\delta},
\end{align}
where parameters $\delta>0$,  $\theta_1>0$, $\theta_2>-1, $ $\theta_3>0$, $\theta_4>-1$,  $\theta_1>\theta_2$, $\theta_3> \theta_4$ and satisfy the following two extra constraints:
\begin{equation}
\label{sameproduct}
\theta_4 [\theta_1+\theta_3(\theta_1-\theta_2)]= \theta_2 [\theta_3+\theta_1(\theta_3-\theta_4)],
\end{equation}
and
$$
(\theta_1-\theta_2)(\theta_3-\theta_4)<1.
$$

In this case, $X$ and $Y$ both have NB marginals with pgf's:
\begin{equation}
\label{negativebinomialmargin}
\mathbb{E}[u^X]= \frac{1}{\Big[1+ \frac{\theta_1 + \theta_3 (\theta_1-\theta_2) }{1-(\theta_1-\theta_2)(\theta_3-\theta_4)} (1-u) \Big]^\delta},  \qquad \mathbb{E}[v^Y]=\frac{1}{\Big[ 1+\frac{\theta_3 + \theta_1(\theta_3-\theta_4) }{1-(\theta_1-\theta_2)(\theta_3-\theta_4)} (1-v)\Big]^\delta},
\end{equation}
whereas the joint pgf is equal to:
\begin{align*}
 \mathbb{E}[u^X v^Y]=\frac{1}{\Big[1+\frac{\theta_1 + \theta_3 (\theta_1-\theta_2) }{1-(\theta_1-\theta_2)(\theta_3-\theta_4)} (1-u)+\frac{\theta_3 + \theta_1(\theta_3-\theta_4) }{1-(\theta_1-\theta_2)(\theta_3-\theta_4)} (1-v)+\underbrace{ \frac{\theta_4 [\theta_1+\theta_3(\theta_1-\theta_2)]}{1-(\theta_1-\theta_2)(\theta_3-\theta_4)}}_{=\frac{ \theta_2 [\theta_3+\theta_1(\theta_3-\theta_4)]}{1-(\theta_1-\theta_2)(\theta_3-\theta_4)}}  (1-u)(1-v) \Big]^\delta}.
\end{align*}
  The conditional expectations are:
\begin{equation}
\label{conditionalexpectation}
\mathbb{E}[X\mid Y]= (\theta_1-\theta_2)Y   + \delta \theta_1, \qquad \mathbb{E}[Y\mid X]= (\theta_3-\theta_4)X   + \delta \theta_3.
\end{equation}

 In the special case where $\theta_2=\theta_4=0$, we recover Example 1.  Thus, this corresponds to an extension of Example 1.

\end{thm}
\begin{proof}
See Appendix \ref{proofprop2}.
\end{proof}


\begin{rem}
 If the distributions of $X$ and $Y$ are identical, then the joint distribution of $(X, Y)$ defines a stationary Markov chain with stationary distribution equal to the marginal distribution of $X$ and $Y$, and transition distribution equal to the conditional distribution of $Y \mid X$.  Moreover, by eq.\eqref{conditionalpgf},  this Markov process $(Y_t)$ is such that:
$$
\mathbb{E}[v^{Y_{t+1}}|Y_t]= \Big(\mathbb{E}[v^{Z_1}]\Big)^{Y_t} \mathbb{E}[v^\eta], \qquad \mathbb{E}[v^{Y_{t}}|Y_{t+1}]= \Big(\mathbb{E}[v^{W_1}]\Big)^{Y_t} \mathbb{E}[v^\epsilon].
$$
Such a process is called compound autoregressive, or affine \citep{darolles2006structural}.  The set of Markov count processes that are compound autoregressive in both directions have been characterized in the literature. More precisely,  \cite[Proposition 2.8]{gourieroux2023noncausal} show that if a count valued stationary Markov process is compound autoregressive in both directions, then it is time reversible.  Thus, the set of stationary, Markov compound autoregressive and reversible processes has been characterized by \cite{gourieroux2023noncausal}, who show that under the assumption that if the process is unbounded, then its conditional distribution in calendar time is necessarily of one of the forms described in Theorem 2. Therefore, our result is an extension of these results, by relaxing the restriction of identical margins for $X$ and $Y$.

\end{rem}
\begin{rem}
The distribution given by \eqref{zeroinflated} can be viewed as an extension of the NB distribution (which arises if $\theta_2=0$).  If $\theta_2$ is positive,  then this distribution is a zero-inflated geometric distribution since
$$
\frac{1+\theta_2(1-u)}{1+\theta_1(1-u)}= \left(1-\frac{\theta_2}{\theta_1}\right)+\frac{\theta_2}{\theta_1}\cdot \frac{1}{1+\theta_1(1-u)}.
$$
If $\theta_2$ is negative, \textit{i.e.},  $\theta_2 \in (-1,0)$,  then
$1+\theta_2(1-u)= (1+\theta_2)+(-\theta_2) u$ is the pgf of the Bernoulli distribution with probability parameter $-\theta_2$.  Thus,  the distribution given by \eqref{zeroinflated} is the convolution of a geometric distribution and a Bernoulli distribution.

Moreover,  eq.\eqref{sameproduct} says that if $\theta_2=0$ then $\theta_4=0$ and vice versa,  since in \eqref{sameproduct}, the other two terms $\theta_1+\theta_3(\theta_1-\theta_2)$ and $\theta_3+\theta_1(\theta_3-\theta_4)$ are positive.  Similarly,  if $\theta_2>0$ (resp.,  $\theta_2<0$),  then we conclude that $\theta_4>0$ (resp.,  $\theta_4<0$) as well.  Thus, the distributions of $Z_i$ and $W_i$ are always of the same type (NB,  zero-inflated NB, or convolution between Bernoulli and NB).
\end{rem}

\begin{rem}[Interpretation using Poisson-gamma conjugacy]
If $\theta_2=\theta_4=0$, then 
we recover Example 1, in which $X, Y$ are conditionally Poisson given $U$,  with $U$ following the gamma distribution.

The case where $\theta_2>0$ (and hence $\theta_4>0$) can be interpreted as follows.  Let us consider the (Markov) chain:
\begin{equation}
\label{causalchain}
X \rightarrow U_1 \rightarrow N \rightarrow U_2 \rightarrow Y.
\end{equation}
In this chain, the conditional distribution of one variable given all previous ones only depends on the immediate previous one, with:
\begin{itemize}
\item given $X$, the continuous variable $U_1$ follows the gamma distribution with shape parameter $\delta+ X$ and rate parameter $\beta_1$,
\item given $U_1$, the count variable $N$ follows the Poisson distribution with parameter $\beta_2 U_1$,
\item given $N$,  the continuous variable $U_2$ follows the gamma distribution with shape parameter $\delta+ N$ and rate parameter $\beta_3$,
\item given $U_2$,  the count variable $Y$ follows the Poisson distribution with parameter $\beta_4 U_2$.
\end{itemize}
Finally,  we assume that the marginal distribution of $U_1$ is  with shape parameter $\delta$ and rate parameter $\beta_0< \beta_1$.  By the Bayes' formula, this implies that $X \mid U_1$ is Poisson with parameter $(\beta_1-\beta_0)  U_1$. Thus, the conditional distributions alternate between the pair of conditional distributions of $(X_1, U)$, which are Poisson and gamma.   This chain structure has been also considered in a time series setting by \cite{pitt2002constructing, pitt2005constructing,  gourieroux2006autoregressive,  gourieroux2019negative}.

Let us check that under model \eqref{causalchain}, the conditional distributions of $X \mid Y$ and $Y \mid X$ are indeed of the form \eqref{zeroinflated} and \eqref{zeroinflated2}.  First, $N$ is NB as a mixed Poisson with gamma mixing variable,  with pmf: $$
\mathbb{P}[N=n]= \frac{\Gamma(n+\delta)}{n! \Gamma(\delta)} \frac{\beta_2^n \beta_0^\delta}{(\beta_2 +\beta_0)^{n+\delta}}= \binom{n+\delta-1}{n} \frac{\beta_2^n \beta_0^\delta}{(\beta_2 +\beta_0)^{n+\delta}}, \qquad \forall n \geq 0.
$$
Then the joint pmf/density of $(N, U_2)$ is:
\begin{equation}
\label{u2n}
f(u_2, n)= \frac{\Gamma(n+\delta)}{n! \Gamma(\delta)} \frac{\beta_2^n \beta_0^\delta}{(\beta_2 +\beta_0)^{n+\delta}} \frac{\beta_3^{\delta+n}}{\Gamma(\delta+n)} u_2^{\delta+n-1}e^{-\beta_3 u_2}, \qquad \forall u \in \mathbb{R}_{>0}, \forall n \in \mathbb{N}.
\end{equation}
Thus, by the Bayes' formula,  the conditional pmf $f(n \mid u_2)$ is the pmf of the Poisson distribution with parameter $\frac{\beta_2 \beta_3}{\beta_0+\beta_2}u_2$.  Hence, the marginal density of $U_2$ is obtained by summing \eqref{u2n} with respect to $n$:
$$
f(u_2)=\sum_{n=0}^\infty f(u_2, n) \propto  u_2^{\delta-1}e^{-\beta_3 u_2} e^{\frac{\beta_2 \beta_3}{\beta_0+\beta_2}u_2},
$$
which is gamma with shape parameter $\delta$ and rate parameter $\frac{\beta_0 \beta_3}{\beta_0+\beta_2}$.  We conclude, by Poisson-gamma conjugacy,  that the marginal distribution of $Y$ is NB and the conditional distribution of $U_2$ given $Y$ is gamma with shape parameter $\delta+Y$ and rate parameter $\frac{\beta_0 \beta_3}{\beta_0+\beta_2}+ \beta_4$.

As a consequence, eq.\eqref{causalchain} becomes:
\begin{align*}
&X\qquad  \underrightarrow{gamma} \qquad U_1 \qquad \underrightarrow{Poisson}  \qquad N \qquad \underrightarrow {gamma} \qquad  U_2 \qquad \underrightarrow{Poisson} \qquad Y \\
&X\qquad  \overleftarrow{Poisson} \qquad U_1 \qquad \overleftarrow{gamma}  \qquad  N \qquad \overleftarrow {Poisson} \qquad  U_2 \qquad \overleftarrow{gamma} \qquad Y.
\end{align*}

Let us now compute the conditional pgf of $Y$ given $X$.  By iterated expectations, we get:
\begin{align*}
\mathbb{E}[v^Y\mid X] &= \mathbb{E}\Big[\mathbb{E}[v^Y\mid U_2, X ] \mid X\Big] \\
&=\mathbb{E}[e^{\beta_4U_2 (v-1)} \mid X] =\mathbb{E}\Big[\mathbb{E}[e^{\beta_4U_2 (v-1)} \mid N] \mid X\Big]\\
&=\mathbb{E}\left[\left.\frac{1}{[1+\frac{\beta_4}{\beta_3}(1-v)]^{\delta+N}} \right| X\right]= \mathbb{E}\left[ \mathbb{E}\left[\left.\left.\frac{1}{[1+\frac{\beta_4}{\beta_3}(1-v)]^{\delta+N}}\right| U_1 \right]\right| X\right]\\
&= \frac{1}{[1+\frac{\beta_4}{\beta_3}(1-v)]^{\delta }} \mathbb{E}\left.\left[\exp \left[\beta_2 U_1 \left(\frac{1}{1+\frac{\beta_4}{\beta_3}(1-v)}-1\right)\right] \right| X\right]\\
&=  \frac{1}{[1+\frac{\beta_4}{\beta_3}(1-v)]^{\delta }}  \cdot\frac{1}{\Big(1+\frac{\frac{\beta_2}{\beta_1}\frac{\beta_4}{\beta_3}(1-v)}{1+\frac{\beta_4}{\beta_3}(1-v)} \Big)^{\delta+X}}.
\end{align*}
Matching with \eqref{zeroinflated}, we get:
$$
\theta_1= \frac{\beta_2}{\beta_1}\cdot\frac{\beta_4}{\beta_3}+ \frac{\beta_4}{\beta_3},\ \ \ \  \theta_2= \frac{\beta_4}{\beta_3}.
$$
Similarly,  by computing $\mathbb{E}[u^ X \mid Y]$ using iterated expectations,  we get a similar expression.
\end{rem}


 \subsection{The higher-dimensional case}

Let us now consider the three-dimensional compound autoregressive model:
 \begin{align}
X\mid Y, Z&=_{(d)} \sum_{i=1}^ Y W_{1i} + \sum_{i=1}^ Z W_{2i} + \epsilon_{1},   \label{car1}\\
Y\mid X, Z&=_{(d)} \sum_{i=1}^ X W_{3i} + \sum_{i=1}^ Z W_{4i} + \epsilon_{2} , \label{car2}\\
Z\mid X, Y&=_{(d)} \sum_{i=1}^ X W_{5i} + \sum_{i=1}^ Y W_{6i} + \epsilon_{3},  \label{car3}
 \end{align}
  where
\begin{itemize}
\item   $(W_{1i})$,  $(W_{2i})$ are i.i.d.  mutually independent,  and independent of $Y, Z, $ and $\epsilon_{1}$,
\item   $(W_{3i})$,  $(W_{4i})$ are i.i.d.  mutually independent,  and independent of $X, Z, $ and $\epsilon_{2}$,
\item   $(W_{5i})$,  $(W_{6i})$ are i.i.d.  mutually independent,  and independent of $X,Y, $ and $\epsilon_{3}$,
\item $\epsilon_{1}$ is independent of $Y, Z$,  $\epsilon_{2}$ is independent of $X, Z$ and $\epsilon_{3}$ is independent of $X, Y$.
\end{itemize}
In other words, model \eqref{car1}-\eqref{car3} can be equivalently written using the conditional pgf's:
\begin{align*}
\mathbb{E}[u^X \mid Y, Z] &= \Big(\mathbb{E}[u^{W_{11}}]\Big)^Y \Big(\mathbb{E}[u^{W_{21}}]\Big)^Z  \mathbb{E}[u^{\epsilon_1}], \\
 \mathbb{E}[v^Y \mid X, Z] &= \Big(\mathbb{E}[v^{W_{31}}]\Big)^X \Big(\mathbb{E}[v^{W_{41}}]\Big)^Z  \mathbb{E}[v^{\epsilon_2}], \\
 \mathbb{E}[t^Z \mid X, Y] &= \Big(\mathbb{E}[t^{W_{51}}]\Big)^X \Big(\mathbb{E}[t^{W_{61}}]\Big)^Y  \mathbb{E}[t^{\epsilon_3}].
\end{align*}

\begin{exm}
	Let us now show that \eqref{car1}-\eqref{car3} (and hence \eqref{pm1}-\eqref{pm3} below) has a solution by extending Example 1. Suppose that $X, Y, Z$ are conditionally Poisson with parameters $\lambda_1 U, \lambda_2U, \lambda_3U$, respectively, where $U$ follows the gamma distribution. Then, the conditional distribution of $X$ given $Y, Z$ is NB distributed with probability parameter $p_1= \frac{\lambda_2+\lambda_3 +
	\beta}{\lambda+1+\lambda_2+\lambda_3+\beta}$, and number of successes parameter $\alpha+y + z$.  Thus, in \eqref{car1}, both $W_{1i}$ and $W_{2i}$ are geometric with probability $p_1$.
\end{exm}

The next theorem says that Example 2 is indeed the unique solution to the three-dimensional compound autoregressive model.
\begin{thm}
The only non-degenerate solution to \eqref{car1}-\eqref{car3}  is given by Example 3.   By
non-degenerate we mean that $W_{1i}, ...,W_{6i}$ should not be constant variables such as zero.
\end{thm}
\begin{proof}
We choose $Z=0$ in eqs.\eqref{car1} and \eqref{car2}, and take marginal expectations on all sides.  We get:
\begin{align*}
\mathbb{E}[u^X \mid Y, Z=0] &= \Big(\mathbb{E}[u^{W_{11}}]\Big)^Y   \mathbb{E}[u^{\epsilon_1}], \\
 \mathbb{E}[v^Y \mid X, Z=0] &= \Big(\mathbb{E}[v^{W_{31}}]\Big)^X    \mathbb{E}[v^{\epsilon_2}].
\end{align*}
Then we can apply Theorem 2 to the joint distribution of $(X,Y)$ given $Z=0$,  and conclude that one of the two following cases arises:
\begin{itemize}
\item either $W_{1i}, W_{3i}$ are  Bernoulli with probability parameters  and $\epsilon_1$, $\epsilon_2$ are Poisson,
\item or the pgf's of $W_{1i}, W_{3i}$ are given by eqs.\eqref{zeroinflated} and \eqref{zeroinflated2}, and $\epsilon_1, \epsilon_2$ are NB.
\end{itemize}
Similarly,  we can choose $Y=0$ in eqs.\eqref{car1} and \eqref{car3} and conclude that $W_{2i}, W_{6i}$, $\epsilon_1, \epsilon_3$ satisfy one of the two conditions of Theorem 2.

Thus, $\epsilon_1, \epsilon_2, \epsilon_3$ are either all Poisson, or all NB.  In the first case,  $W_{11}, ...,W_{61}$ are all Bernoulli; whereas in the second case,  the pgf's of $W_{11}, ...,W_{61}$ are all of the form $\frac{1+\theta_{2j}(1-u)}{1+\theta_{1j}(1-u)}, j=1,...,6$. Let us now discuss these two cases.

\paragraph{Poisson case.}  We denote by $\lambda_i,i=1,2,3$, the Poisson parameters of $\epsilon_i$, and $a_i,  i=1,...,6$, the probability parameters of the Bernoulli sequences $(W_{ij})$, respectively.  Then, we have:
\begin{align}
 &   \mathbb{E}[u^X \mid Y, Z]  = e^{\lambda_1(u-1)}    (a_1 u+1-a_1) ^Y (a_2 u+1-a_2) ^Z, \label{usingu} \\
 & \mathbb{E}[v^Y \mid X, Z]  = e^{\lambda_2(v-1)}  (a_3 v+1-a_3) ^X  (a_4 v+1-a_4) ^Z.\nonumber 
\end{align}
Let us now take $Z=1$ in \eqref{usingu},  and consider the joint distribution of $(X,  Y)$ given $Z=1$.  We get:
\begin{align*}
\mathbb{E}[u^X \mid Y, Z=1] &  = e^{\lambda_1(u-1)}   (a_2 u+1-a_2)   (a_1 u+1-a_1) ^Y,  \\
\mathbb{E}[v^Y \mid X, Z=1] & = e^{\lambda_2(v-1)} (a_4 v+1-a_4)    (a_3 v+1-a_3) ^X.
\end{align*}
This is once again a model of form \eqref{inar2}.  Hence, we can apply Theorem 2 to this joint distribution of $(X,  Y)$ given $Z=1$. We conclude that $ e^{\lambda_1(u-1)}   (a_2 u+1-a_2) $ should either be a Poisson or a NB pgf.  Thus, the only possibility is that $a_2=0$.  Similarly,  $a_4=0$ and $a_1=a_3=a_5=a_6=0$.

\paragraph{NB case.} Let us consider the case:
\begin{align*}
&\mathbb{E}[u^{W_{11}}]= \frac{1+\lambda_1 (1-u)}{1+\theta_1 (1-u)}, \quad \mathbb{E}[u^{W_{21}}]= \frac{1+\lambda_2 (1-u)}{1+\theta_1 (1-u)}, \quad \mathbb{E}[u^{\epsilon_1}]= \frac{1}{[1+\theta_1 (1-u)]^\delta}, \\
&\mathbb{E}[v^{W_{31}}]= \frac{1+\lambda_3 (1-v)}{1+\theta_2 (1-v)}, \quad \mathbb{E}[v^{W_{41}}]= \frac{1+\lambda_4 (1-v)}{1+\theta_2 (1-v)}, \quad \mathbb{E}[v^{\epsilon_2}]= \frac{1}{[1+\theta_2 (1-v)]^\delta}, \\
&\mathbb{E}[t^{W_{51}}]= \frac{1+\lambda_5(1-t)}{1+\theta_3 (1-t)}, \quad \mathbb{E}[t^{W_{61}}]= \frac{1+\lambda_6 (1-u)}{1+\theta_3 (1-u)}, \quad \mathbb{E}[t^{\epsilon_3}]= \frac{1}{[1+\theta_3 (1-t)]^\delta}.
\end{align*}

 Similarly,  by considering the joint distribution of $(X, Y)$ given $Z=1$ we get:
 \begin{align*}
\mathbb{E}[u^X \mid Y, Z=1] &  = \frac{1}{[1+\theta_1 (1-u)]^\delta}\cdot  \frac{1+\lambda_2 (1-u)}{1+\theta_1 (1-u)}  \Big(\frac{1+\lambda_1 (1-u)}{1+\theta_1 (1-u)}\Big) ^Y,  \\
\mathbb{E}[v^Y \mid X, Z=1] & = \frac{1}{[1+\theta_2 (1-v)]^\delta}\cdot \frac{1+\lambda_4 (1-v)}{1+\theta_2 (1-v)}      \Big(\frac{1+\lambda_3 (1-v)}{1+\theta_2 (1-v)}\Big) ^X .
\end{align*}
By Theorem 2,  we conclude that $\frac{1}{[1+\theta_1 (1-u)]^\delta} \cdot \frac{1+\lambda_2 (1-u)}{1+\theta_1 (1-u)} $ is either a Poisson or a NB pgf.  This is only possible when $\lambda_2=0$.  Similarly, $\lambda_i=0$ for any $i=1,...,6$.   Therefore, we recover Example 3.

\end{proof}

\begin{rem}
It is easily shown that if model \eqref{car1}-\eqref{car3} is extended to arbitrary dimension $n$, then the unique non-degenerate solution is given by the Poisson-gamma conjugacy model described in Example 3.  That is,  given $U$,  counts $X_1,...,X_n$ are independent Poisson distributed with parameters $\lambda_i U$, $i=1,...,n$, respectively,  and $U$ is gamma distributed.
\end{rem}

\section{Characterization of solutions to the random coefficient model}
\subsection{The bivariate case}

Let us consider model (\ref{inar2}), where $(Z_i)$ is conditionally i.i.d. given probability parameters $p$ and independent of $\epsilon$, $(W_i)$ is conditionally i.i.d. given probability parameters $q$ and independent of $\eta$, respectively, and $p, q$ are random between 0 and 1, but not necessarily independent. We call this random coefficient autoregressive model, after the time series literature on random coefficient autoregressive count models \citep{zheng2007first}.

 \begin{thm}
 	If the support of $(X, Y)$ is the entire space $\mathbb{N}_0^2$, then model \eqref{inar2} is compatible if and only if
 $p$ and $q$ both follow the beta distribution with parameters $(\alpha, \beta_1)$ and  $(\alpha, \beta_2)$, respectively, and $\epsilon$ (resp.,  $\eta$) follows NB distributions with number of successes parameter $\beta_2$ (resp., $\beta_1$) and common probability parameter $\theta$ between 0 and 1.
 \end{thm}
This compatible model has the stochastic (trivariate NB reduction) representation:
$$
X=_{(d)} Z+ \epsilon, \qquad Y_{(d)}=Z+\eta,
$$
where $Z, \epsilon, \eta$ are independent,  NB distributed, with parameters $(\alpha, \theta)$, $(\beta_1, \theta)$, $(\beta_2, \theta)$. We get:
\begin{equation}
\label{trivariatenb}
\mathbb{E}[X\mid Y]= \frac{\alpha}{\alpha+\beta_2}Y+ \frac{\beta_1 (1-\theta)}{\theta}, \qquad \mathbb{E}[Y\mid X]= \frac{\alpha}{\alpha+\beta_1}X+ \frac{\beta_2 (1-\theta)}{\theta}.
\end{equation}

This result extends Proposition OA.1 of \cite{gourieroux2023noncausal}, who deal with the case where $X, Y$ have equal margins.

\begin{proof}
See Appendix \ref{proofprop4}.
\end{proof}

\subsection{The higher-dimensional case}
Let us now consider the extension of Theorem 3.  We have:
\begin{thm}
Consider the modified model \eqref{car1}-\eqref{car3} in which $(W_{1i})$,  $(W_{2i})$, ..., $(W_{6i})$ are i.i.d. given stochastic probability parameters $p_1, p_2, ..., p_6$.  Then, this model does not have any non-degenerate solution.
\end{thm}
Thus,  the random coefficient model only has solution in dimension 2.  
\begin{proof}
We take $Z=0$ and apply Theorem 3 to the joint distribution of $(X, Y)$ given $Z=0$.  We deduce that $\epsilon_1$, $\epsilon_2$ are NB distributed with the same probability parameter $\pi$, and number of successes parameters $\beta_1$ and $\beta_2$.  Moreover,  the stochastic parameter $p_1$ is beta distributed with parameters $\alpha$ and $\beta_2$, whereas the stochastic parameter $p_3$ is beta distributed with parameters $\alpha$ and $\beta_1$,  respectively.

We then take $Z=1$ and apply Theorem 3 to the joint distribution of $(X, Y)$ given $Z=0$.  We conclude that $\epsilon_1+ W_{21}$ should still be NB, where $W_{21}$ is Bernoulli and $\epsilon_1$ is NB.  This is not possible.  We have arrived at a contradiction.
\end{proof}
\section{Linear conditional expectation model}
We have seen that many of the solutions to the bivariate compound autoregressive or random coefficient autoregressive model fail to extend to dimension 3.  Thus, a natural question is whether we can relax the assumption on the conditional distributions and specify only the conditional expectations,  in order to allow for more flexible joint distributions.  

 \subsection{Some examples}

It is well known that if the conditional distributions of $X\mid U$ and $Y\mid U$ belong to the same linear exponential family and if the distribution of $U$ is conjugate prior, then under certain regularity conditions, the conditional expectations $\mathbb{E}[Y\mid X]$ and $\mathbb{E}[X\mid Y]$ are all linear.  This is for instance the case in Example 1, which is also extended in Example 2.  The following Example 3 discusses another popular conjugacy for count distributions.

\begin{exm}[Beta-NB conjugacy]
Let $U$ be beta$(\alpha_1, \alpha_2)$ distributed. Assume that, given $U$, $X$ and $Y$ are NB distributed with parameters $(r_1, p)$ and $(r_2, p)$, respectively.  Then 
by the Bayes' formula, the posterior distribution of $U\mid X=x$ has the density:
$$
\pi(u|x) =\frac{\ell(x|u) \pi(u)}{\ell(x)} \propto \binom{x+r_1-1}{x} (1-u)^x u^{r_1} u^{\alpha_1-1}(1-u)^{\alpha_2-1},
$$
which is beta with parameters $\alpha_1+r_1$ and  $\alpha_2+x$. Thus, the conditional distribution of $Y \mid X$
is beta-NB distributed with number of successes parameter $r_2$ and beta parameters $\alpha_1+r_1$ and  $\alpha_2+X$. Similarly, the conditional distribution of $X \mid Y$ is also beta-NB distributed with number of successes parameter $r_1$ and beta parameters $\alpha_2+r_2$ and  $\alpha_1+Y$.

By the iterated expectation formula, we get:
\begin{equation}\label{betanb}
\mathbb{E}[Y\mid X]=\frac{r_2 (\alpha_2+X)}{\alpha_1+r_1-1}, \qquad  \mathbb{E}[X\mid Y]=\frac{r_1 (\alpha_2+Y)}{\alpha_1+r_2-1},
\end{equation}
so long as $\alpha_1+r_1> 1$ and ${\alpha_1+r_2>1}$. 

\end{exm}

Similarly to Example 2, we can extend Example 3 to the higher dimensional case. 

The following two examples involve solutions with bounded support.
\begin{exm}
	If $(X_1,...,X_n)$ is a joint mix, that is,  $X_1+\cdots+X_n$ is a constant, say $M$,  then the joint distribution satisfies:
 \begin{align}
 \mathbb{E}[X \mid Y,  Z] &= a_{11} +a_{12} Y + a_{13} Z, \label{pm1}\\
 \mathbb{E}[Y \mid X,  Z] &= a_{22} +a_{21} X + a_{23} Z , \label{pm2}\\
  \mathbb{E}[Z \mid X, Y] &= a_{33} +a_{31} X + a_{32} Y,   \label{pm3}
 \end{align}
with regression coefficients all equal to $-1$.  An example of jointly mix distribution is, for instance, the multinomial distribution.
\end{exm}
This example, however, involves a multivariate distribution whose support is not $\mathbb{N}_0^n$.  Moreover, the regression coefficients are all negative and equal to $-1$.

\begin{exm}
Let us assume that $(Z_1,  Z_2, Z_3)$ follows the three-dimensional multinomial distribution with size parameter $n$ and probability parameters $p_1,  p_2,  p_3$ such that $p_1+p_2+p_3=1$ and $Z_1+Z_2+Z_3=n$,  and consider the joint distribution of $X=Z_1,  Y=Z_1+Z_3$.  Then, we have:
\begin{align*}
\mathbb{E}[X \mid Y ]&= \frac{p_1}{p_1+p_3}Y, \\
  \mathbb{E}[Y \mid X]&= \mathbb{E}[n-Z_2 \mid Z_1] =n-\mathbb{E}[Z_2 \mid Z_1]=n-\mathbb{E}[Z_2 \mid Z_2+Z_3=n-X]= n-\frac{p_2}{p_2+p_3}(n-X)\\
  &= \frac{p_2}{p_2+p_3}X+\frac{p_3}{p_2+p_3}n.
\end{align*}
This example is such that one of the constants in the linear regressions is equal to zero.
\end{exm}

\subsection{Some counter-examples}
According to Theorems 3 and 4,  some extensions of the bivariate distributions found to solve either the compound autoregressive  or the random coefficient model may no longer solve the respective model in the three-dimensional case.  Is it nevertheless possible that they satisfy at least the linear conditional expectation condition? In this subsection, we show that this is in general not the case.
\begin{exm}Consider the following extension of Example 2:
\begin{align}
X&= W_1+W_{12}+W_{13}, \label{w1} \\
Y&= W_2+ W_{12}+W_{23},  \label{w2}\\
Z&= W_3+ W_{13}+W_{23}, \label{w3}
\end{align}
where $W_1, W_2, W_3, W_{12}, W_{13}, W_{23}$ are independent Poisson variables with positive parameters $\lambda_1, \lambda_2, \lambda_3, \lambda_{12}, \lambda_{13}, \lambda_{23}$, respectively.

Alternatively, we also consider the model:
 \begin{align}
	X&= W_1+W_{123}, \label{w123}\\
	Y&= W_2+ W_{123}, \label{w123bis}\\
	Z&= W_3+W_{123},\label{w123ter}
\end{align}
where $W_1, W_2, W_3, W_{123}$ are independent Poisson distributed with parameters $\lambda_1, \lambda_2, \lambda_3, \lambda_{123},$ respectively.  
\end{exm}

It is straightforward to show that in both models \eqref{w1}-\eqref{w3} and \eqref{w123}-\eqref{w123ter},  pairwisely, $(X, Y)$, $(Y, Z)$, $(Z, X)$ all have the trivariate Poisson reduction representation,  and thus have linear pairwise conditional expectations.  The following lemma says that jointly, though,  the distribution of $(X, Y, Z)$ does not have linear conditional expectation.
\begin{lm}
\label{ex6}
In both models \eqref{w1}-\eqref{w3} and \eqref{w123}-\eqref{w123ter},  the conditional expectations,  such as $\mathbb{E}[X\mid  Y, Z]$,  are non-linear in conditioning variables except in the mutually independent case.
\end{lm}
\begin{proof}
See Appendix \ref{proofex6}.
\end{proof}

The same technique can be used to show that if we extend the bivariate model found in Theorem 4,  that is by assuming that $W_1, W_2, W_3, W_{12}, W_{13}, W_{23}$ are independent NB distributed with the same probability parameter,  then the resulting model also does not allow for linear conditional expectations.

\begin{exm}
Let us now consider the joint distribution of $X, Y$ and $Z=N$ in model \eqref{causalchain}.  We have shown that $N$ marginally follows a NB distribution with number of successes parameter $\delta$ and probability parameter $p_0$,  say, and by the Markov property,  $X, Y$ are independent given $N$,  with conditional NB distribution  with the same number of successes parameters $\delta+N$ and probability parameters $p_1, p_2$, say, respectively.  Thus, the conditional expectations $\mathbb{E}[X \mid N, Y]$ and $\mathbb{E}[Y\mid N, X]$ all depend on $N$ only and are linear in $N$.  The next lemma says that $\mathbb{E}[N \mid X, Y]$ is not linear.
 \end{exm}

 \begin{lm}
\label{marcheplusdimension3}
 In model \eqref{causalchain},  the joint distribution of $X, Y$ and $Z=N$  is such that $\mathbb{E}[N \mid X, Y]$ is non-linear in $X$ and $Y$.
 \end{lm}

 \begin{proof}
 See Appendix \ref{proofmarcheplusdimension3}.
 \end{proof}

 \subsection{New solutions through linear programming}

Let us first work out the necessary conditions on the regression coefficients in model \eqref{linearconditionalexp}.
\begin{lm}
\label{principalminor2}
If model \eqref{linearconditionalexp} has a solution $(X,Y)$ with positive variances, then $ac=[{\rm Corr}(X,Y)]^2$ and one of the following cases is satisfied:

\noindent (i) $a=c=0$ and $X, Y$ are uncorrelated.

\noindent (ii) $ac\in(0,1)$.

\noindent (iii) $ac=1$ and $Y=aX+b$,  that is,  $X$ and $Y$ are linearly dependent.
\end{lm}

 \begin{proof}
 See Appendix \ref{prooflemma1}.
 \end{proof}

By virtue of Lemma \ref{principalminor2}, given $a,b,c, d>0$ and $ac<1$,  can we always find a joint distribution $
(X, Y)$ satisfying \eqref{linearconditionalexp}?  This remains an open problem but in the remainder of this subsection, we will provide some partial answers.

First,  most of the bivariate distributions that we have examined in section 3 and 4 that solve the compound autoregressive or the random coefficient model involve some sort of common factor,  which leads to deterministic relationship between $a, b, c, d$.  For instance,
\begin{itemize}
\item in the trivariate Poisson reduction model of Corollary 2,  $a, b, c, d$ can be reparameterized using the three Poisson parameters $\lambda_0,  \lambda_1$ and $\lambda_2$ (see Remark 1).  Moreover,  this model  requires both $a$ and $c$ to be smaller than 1.
\item in the Poisson-gamma conjugacy (see \eqref{poissnogamma} in Example 1), we can check that $a/b=c/d$.  Moreover,  this model  requires both $a$ and $c$ to be smaller than 1.
\item in the trivariate NB reduction example of Theorem 4 (see \eqref{trivariatenb}),  we can check that $\frac{b}{d}=\frac{(1-a)c}{a(1-c)}$.  Moreover,  this model  requires both $a$ and $c$ to be smaller than 1.
\item in the beta-NB (see \eqref{betanb} in Example 3),  we can check that $a/b=c/d$.
\end{itemize}

The only example without deterministic relationship between $a, b, c, d$ is case $ii)$ of Theorem 2.  For this model, the domain of $(a, b, c, d)$ is as follows:
\begin{thm}
\label{abcde}
In model \eqref{zeroinflated}-\eqref{conditionalexpectation},  the domain of $a, b, c, d$ is composed of five regions:

\noindent (a) $a\ge1>c,ac<1$, $ \frac{b}{d} >\sqrt{\frac{a}{c}}$.

\noindent (b) $c\ge1>a,ac<1$, $ \frac{b}{d} <\sqrt{\frac{a}{c}}$.

\noindent (c) $a=c<1$, $ \frac{b}{d} =\sqrt{\frac{a}{c}}=1$.

\noindent (d) $c<a<1$, $1<\sqrt{\frac{a}{c}}<\frac{b}{d}<\frac{a(1-c)}{c(1-a)}$.

\noindent (e) $a<c<1$, $1>\sqrt{\frac{a}{c}}>\frac{b}{d}>\frac{a(1-c)}{c(1-a)}$.

\end{thm}
Thus,  the ratio $\frac{b}{d}$ is partially constrained in model \eqref{zeroinflated}-\eqref{conditionalexpectation}.
\begin{proof}
See Appendix \ref{prooflemma3}.
\end{proof}

Nevertheless,  if we further constraint $a, c$ to be smaller than 1,  then it possible to find other solutions to \eqref{linearconditionalexp} with both $X, Y$ having bounded domains.  More precisely, we have the following result: 
\begin{thm}
\label{farkas}
Suppose that $0<a,c<1$ and $b, d>0$,  then we can find an integer $N$ such that model \eqref{linearconditionalexp} has a solution with $0 \leq X, Y \leq N$.
\end{thm}
Note that the condition $0<a,c<1$ is satisfied in both the trivariate Poisson reduction model of Corollary 2 and the trivariate NB model of Theorem 4. Compared to Theorem \ref{abcde}, in Theorem \ref{farkas},  the ratio $b/d$ is no longer constrained,  if both $a$ and $c$ are smaller than 1.  In other words,  when we allow for bounded discrete variables, model \eqref{linearconditionalexp} can allow for extra solutions.

To get an intuition of this Theorem,  let us remark that under the assumption that $X,Y\in\{0,1,\dots,N\}$,  model \eqref{linearconditionalexp} becomes
\begin{eqnarray}
&&\sum_{i,j=0}^Np_{ij}=1,\label{farka1}\\
&&\sum_{j=0}^Njp_{ij}=(ai+b)\sum_{j=0}^Np_{ij},\ \ \ \ i=0,\dots,N \Longleftrightarrow \sum_{j=0}^N(ai+b-j)p_{ij}=0 ,\ \ \ \ i=0,\dots,N,\label{farka2}\\
&& \sum_{i=0}^Nip_{ij}=(cj+d)\sum_{i=0}^Np_{ij},\ \ \ \ j=0,\dots,N \Longleftrightarrow \sum_{i=0}^N(cj+d-i)p_{ij}=0,\ \ \ \ j=0,\dots,N. \label{farka3}
\end{eqnarray}
 This is a linear system of equations under the inequality constraints that $$p_{ij} \geq 0,\ \ \ \  i, j=0,\dots,N. $$ The solvability of this kind of systems can be checked using a classical linear programming result called Farkas' lemma [see e.g. \cite{dax1997classroom} and section 5.8.3 of \cite{boyd2004convex}]:
 \begin{thm}[Farkas' lemma] Let $A\in \mathbb{R}^{m\times n}$ and $b\in\mathbb{R}^m$. Then, exactly one of the following two assertions is true:

\noindent (a) There exists an $x\in \mathbb{R}^n$ such that $Ax=b$ and $x\ge0$.

\noindent (b) There exists a $y\in \mathbb{R}^m$ such that $A^Ty\ge0$ and $b^Ty<0$.
\end{thm}
 This means that whether or not a linear system of equations with inequality constraints (such as case $(a)$ in the theorem) can be alternatively checked by inspecting whether a solution to the system of inequalities $(b)$ can exist.  For expository purpose,  the proof that \eqref{farka1}-\eqref{farka3} allows for a solution is relegated to the Appendix.

Whether or not Theorem \ref{farkas} can be extended to include the case where either $a>1$ or $c>1$ remains an open problem.  The main difficulty is that, if say  $a\ge1$, then any solution to \eqref{linearconditionalexp} must satisfy the condition that $\mathbb{P}(X=N)=0$.  Indeed,  $\mathbb{P}(X=N)>0$ implies $\mathbb{E}[Y \mid X=n]=a n+b>n$,  which contradicts the assumption that $Y \leq N$ almost surely.

\subsection{The higher-dimensional case}
For $n=3$,  let us consider the linear conditional expectation model \eqref{pm1}-\eqref{pm3},  where coefficients $a_{ij}, i,j=1,2,3$, are non-negative.

Clearly,  all the conjugate prior examples, such as Examples 1 and 3, can be extended to higher dimensions.  One downside of these models, though, is that these models usually imply equality of the regression coefficients for each conditional expectation.  For instance,  for $n=3$,  a model based on conjugate prior would lead to $a_{12}=a_{13}, a_{21}=a_{23}$ and $a_{31}=a_{32}$, such that $\mathbb{E}[X \mid Y, Z]$ depends on $Y, Z$ only through the sum $Y+Z$.  This result is straightforward to extend to any higher dimensions,  and is due to the fact that for linear exponential families,  the sum of the observations is a sufficient statistic.

Is it possible to find alternative, non-conjugate based solutions such that the coefficients $a_{ij}$ are less constrained?  We have the following extension of Lemma 3.
\begin{thm}
\label{principalminor}
If model \eqref{pm1}-\eqref{pm3} has a solution $(X,Y,Z)$ with finite variances and $X, Y, Z$ are linearly independent,  then

\noindent (i) for all the three pairs $(X, Y)$, $(Y, Z)$, $(Z, X)$,  the pairwise conditional expectations $\mathbb{E}[X \mid Y]$,  $\mathbb{E}[X \mid Z]$, $\mathbb{E}[Y \mid X]$,  $\mathbb{E}[Y \mid Z]$, $\mathbb{E}[Z \mid X]$, $\mathbb{E}[Z \mid Y]$ are also linear,  and
\begin{equation}\label{Feb22s1}
1-a_{12} a_{21}>0,  \quad 1-a_{23} a_{32}>0,  \quad 1-a_{13} a_{31}>0.
\end{equation}
\noindent (ii)
\begin{equation}\label{Feb22s2}
\begin{vmatrix}
1 & -a_{12}  & -a_{13} \\
-a_{21}&1 & -a_{23} \\
-a_{31}&   -a_{32}  & 1
\end{vmatrix} >0.
\end{equation}
That is,  all the principal minors of matrix
$
\begin{bmatrix}
1 & -a_{12} & -a_{13} \\
-a_{21}&1 & -a_{23} \\
-a_{31}&  - a_{32}  & 1
\end{bmatrix}
$ are positive.

More generally,  in arbitrary dimension $n \geq 3$,  if $X=(X_1,\dots,X_n)$ is a linearly independent random vector such that each $X_i$ has finite second moment and such that the conditional expectations are all linear:
\begin{eqnarray*}
\mathbb{E}[X_i|X_j,j\not=i]=a_{ii}+\sum_{j\not=i}a_{ij}X_j,\ \ \ \ 1\le i\le n.
\end{eqnarray*}
Then, all the principal minors of matrix $\begin{bmatrix}
1 & -a_{12} & \cdots &-a_{1n} \\
-a_{21}& 1 & \cdots & \cdots \\
\cdots & \cdots & \cdots  & \cdots \\
-a_{n1} & \cdots & \cdots & 1
\end{bmatrix}$ are positive,  and for any subset $(k_1,...,k_l)$ of $\{1,...,n\}$,  the joint distribution of the subvector $(X_{k_1},...,X_{k_l})$ is such that all conditional expectations are linear.
\end{thm}
\begin{proof}
See Appendix \ref{prooflemma4}. 
\end{proof}

Here,  condition \eqref{Feb22s1} is similar to the condition of Lemma 3 and concerns pairwise distributions $(X, Y)$, $(Y, Z)$ and $(Z, X)$,  whereas condition \eqref{Feb22s2} is an extra constraint on the joint distribution $(X, Y, Z)$.  
This result means that as the dimension increases,  it becomes more difficult to find solutions.  For instance,  Examples 6 and 7 above are such that $X, Y, Z$ all allow for pairwise linear conditional expectations,  but do not satisfy jointly model \eqref{pm1}-\eqref{pm3}.
%


\section{Conclusion}
In this paper, we have analyzed the compatibility of conditional specifications of multivariate count models that lead to linear conditional expectations.  We have completely characterized all the solutions to the compound autoregressive model and the random coefficient integer autoregressive model in any dimensions.  Finally,  since there are only a smaller number of solutions to these two models,  we also investigated the more general case in which only the conditional expectations are assumed to be linear.  We found that this ``semi-parametric" specification is more flexible,  in the sense that they allow for more solutions.

Our findings have important implications for the statistical modeling of count data, especially for the recent literature on spatial count data \citep{glaser2022spatial, karlis2024multilateral}: $i)$ Many of the recently proposed conditional Poisson model with linear conditional expectation,  or the integer autoregressive model,  may not be compatible \textit{per se},  at least in high dimensions.  As a consequence, such models should be interpreted with precaution.  For instance, they should not be used in a Gibbs sampler to simulate from the joint distribution of $(X_1,...,X_n)$,  since this latter might not even exist.  $ii)$ On the other hand,  since the linear conditional expectation model is much more flexible,  this type of specification could be a strong competitor of the state-of-art spatial count models based on high-dimensional latent factors \citep{aitchison1989multivariate,  diggle1998model, chib2001markov, macnab2003hierarchical,  wakefield2007disease, de2013hierarchical} or random field \citep{roy2020nonparametric}.  Indeed,  the parameters of the linear conditional expectations can be estimated much more easily by (possibly constrained) least square,  or pseudo maximum likelihood estimation with respect to a ``mis-specified" likelihood function such as Poisson likelihood \citep{besag1974spatial, gourieroux1984pseudo,  wooldridge1999quasi}.  This would be much less computer intensive than the typical MCMC methods employed into spatial models with latent factors.

Finally,  the linear type conditional expectation is only the first step towards introducing more flexible,  nonlinear conditional specifications, called dependency network in the machine learning literature  \citep{heckerman2000dependency,  allen2013local, bengio2014deep, hadiji2015poisson}.  This literature has so far largely focused on the algorithm side, but have not addressed the compatibility issue.  This is left for future research.

\bibliographystyle{apalike}
\bibliography{lib}

\appendices
\appendix

\section{Proofs }
\subsection{Proof of Theorem 2}
\label{proofprop2}
We follow the proof of \cite{darolles2006structural}.  The conditional Laplace transforms for eq.\eqref{inar2} are given by:
 $$
\mathbb{E}[e^{-uX}\mid Y] =\mathbb{E}[e^{-a(u)Y-b(u)}], \qquad
\mathbb{E}[e^{-vY}\mid X] =\mathbb{E}[e^{-c(u)X-d(u)}],
$$
where $e^{-a(u)}, e^{-b(u)}, e^{-c(u)}, e^{-d(u)}$ are the Laplace transforms of $Z_i, \epsilon, W_i$ and $\eta$, respectively.  Thus, we can calculate the marginal Laplace transforms $D_1$, $D_2$ in two ways:
\begin{align}
e^{-D_1(u)}:=\mathbb{E}[e^{-uX}]=\mathbb{E}[e^{-a(u)Y-b(u)}]=e^{-D_2(a(u))-b(u)}, \label{marginx} \\
e^{-D_2(v)}:=\mathbb{E}[e^{-vY}]=\mathbb{E}[e^{-c(v)X-d(v)}]=e^{-D_1(c(v))-d(v)}, \label{marginy}
\end{align}
as well as the joint Laplace transform:
\begin{align}
 \mathbb{E}[e^{-uX-vY}]&=\mathbb{E}[e^{-a(u)Y-b(u)-vY}]=e^{-D_2(a(u)+v)-b(u)}, \label{joint1} \\
 \mathbb{E}[e^{-uX-vY}]&=\mathbb{E}[e^{-c(v)X-d(v)-uX}]=e^{-D_1(c(v)+u)-d(v)}. \label{joint2}
\end{align}
Equating eqs.\eqref{joint1} and \eqref{joint2}, we get:
$$
D_2(a(u)+v)+b(u)=D_1(c(v)+u)-d(v).
$$
Then, using eqs.\eqref{marginx} and \eqref{marginy} to eliminate $b(u)=D_1(u)-D_2(a(u))$ and $d(v)=D_2(v)-D_1(c(v))$, we get:
\begin{equation*}
D_2(a(u)+v)+ D_1(u)-D_2(a(u))= D_1(c(v)+u)+D_2(v)-D_1(c(v)).
\end{equation*}
Rearranging all the terms depending on $D_2$ (resp., $D_1$) on the LHS (resp., RHS), and using the fact that $D_2(0)=D_1(0)=1$ by definition of the Laplace transform, we get:
$$
[D_2(a(u)+v)-D_2(a(u))]-[D_2(v)-D_2(0)]= [D_1(c(v)+u)-  D_1(u)]-[D_1(c(v))-D_1(0)].
$$

Under the assumption that $X$ and $Y$ have finite moments of all integer orders,  applying Taylor's expansion on both sides of this equality, we get:
\begin{equation}
\label{infinitetaylor}
\sum_{n=1}^\infty \frac{1}{n!} \Big[ \frac{\mathrm{d}^n D_2}{\mathrm{d}  u^n}(a(u)) - \frac{\mathrm{d}^n D_2}{\mathrm{d}  u^n}(0) \Big] v^n = \sum_{n=1}^\infty \frac{1}{n!} \Big[  \frac{\mathrm{d}^n D_1}{\mathrm{d}  u^n}(u) - \frac{\mathrm{d}^n D_1}{\mathrm{d}  u^n}(0) \Big] c(v)^n .
\end{equation}
Similarly, $c(v)$ is Taylor expandable at $v=0$.

Let us now equate the coefficients in front of $v$ in eq.\eqref{infinitetaylor}. We get:
$$
\frac{\mathrm{d} D_2}{\mathrm{d}  u}(a(u)) -\frac{\mathrm{d} D_2}{\mathrm{d}  u}(0) = \frac{\mathrm{d}c}{\mathrm{d}u}(0) \Big[ \frac{\mathrm{d} D_1}{\mathrm{d}  u}(u) -\frac{\mathrm{d} D_1}{\mathrm{d}  u}(0) \Big].
$$
Thus, if $\frac{\mathrm{d} D_2}{\mathrm{d}  u}$ is invertible, we can work out $a(u)$:
\begin{equation}
\label{au}
a(u)=\left(\frac{\mathrm{d} D_2}{\mathrm{d}  u}\right)^{-1} \Bigg [\frac{\mathrm{d} D_2}{\mathrm{d}  u}(0) + \frac{\mathrm{d}c}{\mathrm{d}u}(0) \Big[ \frac{\mathrm{d} D_1}{\mathrm{d}  u}(u) -\frac{\mathrm{d} D_1}{\mathrm{d}  u}(0) \Big] \Bigg],
\end{equation}
where $(\frac{\mathrm{d} D_2}{\mathrm{d}  u})^{-1}$ denotes the inverse of $\frac{\mathrm{d} D_2}{\mathrm{d}  u}$. Similarly, we can equate  the coefficients in front of $v^2$ in eq.\eqref{infinitetaylor}. We get:
$$
\frac{\mathrm{d}^2 D_2}{\mathrm{d}  u^2}(a(u)) -\frac{\mathrm{d}^2 D_2}{\mathrm{d}  u^2}(0) = \left[ \frac{\mathrm{d}c}{\mathrm{d}u}(0)\right]^2 \Big[\frac{\mathrm{d}^2 D_1}{\mathrm{d}  u^2}(u) - \frac{\mathrm{d}^2 D_1}{\mathrm{d}  u^2}(0) \Big] +  \frac{\mathrm{d}^2c}{\mathrm{d}u^2}(0) \Big[  \frac{\mathrm{d} D_1}{\mathrm{d}  u}(u) - \frac{\mathrm{d} D_1}{\mathrm{d}  u}(0) \Big].
$$

Let us introduce the functions: $$\gamma_1(u)=\frac{\mathrm{d}^2 D_2}{\mathrm{d}  u^2} \circ \left(\frac{\mathrm{d} D_2}{\mathrm{d}  u}\right)^{-1} (u) , \qquad \gamma_2(u)=\frac{\mathrm{d}^2 D_1}{\mathrm{d}  u^2} \circ \left(\frac{\mathrm{d} D_1}{\mathrm{d}  u}\right)^{-1} (u),$$
and the change of variable:
$$
w_1(u)=\frac{\mathrm{d} D_1}{\mathrm{d}  u}(u).
$$
Then, eq.\eqref{au} can be written as:
\begin{equation}
\label{firstone}
\gamma_1 \Big[ \alpha_1 w_1+ \alpha_2 \Big]= \alpha_1^2 w_1+\alpha_3 \gamma_2(w_1)+\alpha_4, \qquad \forall w_1,
\end{equation}
where $\alpha_1,...,\alpha_4$ are constants with $\alpha_1 \neq 0$.

Similarly, we can change the roles of $u$ and $v$ in eq.\eqref{infinitetaylor} and get:
$$
\sum_{n=1}^\infty \frac{1}{n!} \Big[ \frac{\mathrm{d}^n D_1}{\mathrm{d}  u^n}(c(v)) - \frac{\mathrm{d}^n D_1}{\mathrm{d}  u^n}(0) \Big] u^n = \sum_{n=1}^\infty \frac{1}{n!} \Big[  \frac{\mathrm{d}^n D_2}{\mathrm{d}  u^n}(v) - \frac{\mathrm{d}^n D_2}{\mathrm{d}  u^n}(0) \Big] a(u)^n.
$$
Similar arguments leads to:
\begin{equation}
\label{secondone}
\gamma_2 \Big[ \beta_1 w_2+ \beta_2 \Big]= \beta_1^2 w_2+\beta_3 \gamma_1(w_2)+\beta_4, \qquad \forall w_2,
\end{equation}
where $\beta_1,...,\beta_4$ are constants with $\beta_1 \neq 0$ and we have used the change of variable $ w_2(v)=\frac{\mathrm{d} D_1}{\mathrm{d}  u}(v)$.

Combining eqs.\eqref{firstone} and \eqref{secondone}, by choosing $w_1=\beta_1 w_2+ \beta_2 $, we get:
$$
\gamma_1 \Big[ \alpha_1 (\beta_1 w_2+ \beta_2)+ \alpha_2 \Big]=\alpha_1^2 (\beta_1 w_2+ \beta_2)+ \alpha_3 \Big[\beta_1^2 w_2+ \beta_3 \gamma_1(w_2)+\beta_4\Big]+ \alpha_4, \qquad \forall w_2.
$$
Hence, we deduce that $\gamma_1$ is a quadratic function.  Similarly, $\gamma_2$ is also quadratic.   Since $\gamma_1(u)=\frac{\mathrm{d}^2 D_2}{\mathrm{d}  u^2} \circ (\frac{\mathrm{d} D_2}{\mathrm{d}  u})^{-1} (u)$, we deduce that $D_2$ solves a Ricatti ordinary differential equation (ODE) [see e.g.  \cite[Proposition 10]{darolles2006structural}]:
$$
\frac{d^2}{d u^2}D_2(u)= k_0+ k_1 \frac{d}{d u}D_2(u)+ k_2 \Big(\frac{d}{d u}D_2(u) \Big)^2.
$$
This Ricatti ODE is solved in Appendix A.1 of \cite{gourieroux2019noncausal}. Taking into account the fact that $X$ is count valued, we get two cases: either $D_2(u)$ is of the log Laplace transform of a Poisson distribution or the log Laplace transform of the NB distribution.  Moreover, in the first case,  both $Z_i$ and $W_i$ are Bernoulli, and $\epsilon, \eta$ are Poisson, whereas in the second case, the pgf's of $Z_i$ and $W_i$ are given by \eqref{zeroinflated} and \eqref{zeroinflated2},  and $\epsilon, \eta$, $X, Y$ are all NB distributed with the same number of successes parameter.

\paragraph{The constraints on the parameters in the Poisson/binomial case.}
These constraints are given by Corollary 2.
%
%

\paragraph{The constraints on the parameters in the NB case.}  Similarly,  by computing the marginal expectations of $X$ and $Y$ in two ways,  we get:
$$
\theta_X \delta = \theta_Y \delta (\theta_1-\theta_2)+ \delta \theta_1,  \qquad \theta_Y \delta = \theta_X \delta (\theta_3-\theta_4)+ \delta \theta_3,
$$
where $\theta_X=\frac{1-p_X}{p_X}$, $\theta_Y=\frac{1-p_Y}{p_Y}$, and $p_X, p_Y$ are the probability parameters of the NB distributions of $X$ and $Y$, respectively.  Solving this linear system for $(\lambda_X, \lambda_Y)$ leads to
$$
\lambda_X=\frac{\theta_1+\theta_3(\theta_1-\theta_2)}{1-(\theta_1-\theta_2)(\theta_3-\theta_4)}, \qquad \lambda_Y= \frac{\theta_3+\theta_1(\theta_3-\theta_4)}{1-(\theta_1-\theta_2)(\theta_3-\theta_4)},
$$
which is \eqref{negativebinomialmargin}.

 By computing the joint pgf in two ways,  we get:
 \begin{align*}
\mathbb{E}[u^X v^Y]&= \mathbb{E}[ \Big(\frac{1+\theta_4(1-v)}{1+\theta_3(1-v)}\Big)^X \frac{u^X}{[1+\theta_3(1-v)]^\delta}] = \frac{1}{[1+\theta_3(1-v)]^\delta} \frac{1}{\Big[ 1+ \frac{\theta_1+\theta_3(\theta_1-\theta_2)}{1-(\theta_1-\theta_2)(\theta_3-\theta_4)}[1-u\frac{1+\theta_4(1-v)}{1+\theta_3(1-v)}]\Big]^\delta}, \\
\mathbb{E}[u^X v^Y]&= \mathbb{E}[ \Big(\frac{1+\theta_2(1-u)}{1+\theta_1(1-u)}\Big)^Y \frac{v^Y}{[1+\theta_1(1-u)]^\delta}] = \frac{1}{[1+\theta_1(1-u)]^\delta} \frac{1}{\Big[ 1+ \frac{\theta_3+\theta_1(\theta_3-\theta_4)}{1-(\theta_1-\theta_2)(\theta_3-\theta_4)}[1-v\frac{1+\theta_2(1-u)}{1+\theta_1(1-u)}]\Big]^\delta}.
\end{align*}
Matching the coefficients in front of $uv$ in the denominator leads to \eqref{sameproduct}.

\subsection{Proof of Theorem 4}
\label{proofprop4}
We have:
\begin{align*}
&\ell(x\mid y)=\sum_{i=0}^{\min(x,y)} \binom{y}{i} \mathbb{E}[p^i (1-p)^{y-i}] g(x-i),\\
&\ell(y\mid x)=\sum_{i=0}^{\min(x,y)} \binom{x}{i} \mathbb{E}[q^i (1-q)^{x-i}] h(y-i) ,
\end{align*}
where $p, q$ are random probability parameters and the expectations $ \mathbb{E}[q^i (1-q)^{x-i}],  \mathbb{E}[p^i (1-p)^{x-i}]$ are with respect to their distributions,  and $g$ and $h$ are the pmf's of  $\epsilon$ and $\eta$, respectively.

By the same method as in Corollary 2, we have:
\begin{equation}
\label{aftersimplification}
\displaystyle \frac{\displaystyle \sum_{i=0}^{\min(x,y)} \binom{y}{i} \mathbb{E}[p^i (1-p)^{y-i}] g(x-i)}{\displaystyle\sum_{i=0}^{\min(x,y)} \binom{x}{i} \mathbb{E}[q^i (1-q)^{x-i}] h(y-i)}= \frac{\mathbb{E}[(1-q)^{y}]g(x)}{\mathbb{E}[(1-q)^{x}]h(y)}.
\end{equation}
Let us define, for any $ i=0,...,x, j=0,...,y$:
$$A(i,x)=\binom{x}{i} \mathbb{E}[q^i (1-q)^{x-i}]/\mathbb{E}[ (1-q)^{x}], \qquad  B(j,y)=\binom{y}{j} \mathbb{E}[p^j (1-p)^{y-j}]/\mathbb{E}[  (1-p)^{y}],
$$
and
$$
H(j,y)=\frac{h(y-j)}{h(y)}, \qquad G(i,x)=\frac{g(x-i)}{g(x)}.
$$
Then, eq.\eqref{aftersimplification} becomes:
\begin{equation}
\label{reducedform}
\displaystyle\sum_{i=0}^{\min(x,y)} A(i,x)H(i,y)= \sum_{i=0}^{\min(x,y)} B(i,y)G(i,x).
\end{equation}

Taking $x=0, y \geq 0$, we get:
$$
A(0,0) H(0, y)=B(0, y) G(0,0),\ \ \ \ \forall y.
$$
Similarly, by taking $y=0, x \geq 0$, we get:
\begin{equation}
\label{ga}
B(0,0)G(0,x)=A(0,x)H(0,0),\ \ \ \ \forall x.
\end{equation}
Similarly, by taking $x=1, y \geq 1$, we get:
$$
A(0,1)H(0, y)+ A(1,1)H(1,y)= G(0,1)B(0,y)+G(1,1)B(1,y).
$$
But we already have $A(0,1)H(0, y)=G(0,1)B(0,y)$ by taking $x=1$ in \eqref{ga}.  Thus, we deduce that:
\begin{equation}
\label{part1}
A(1,1)H(1,y)= G(1,1)B(1,y) \Leftrightarrow \frac{H(1,y)}{B(1,y)}=\frac{G(1,1)}{A(1,1)}, \qquad \forall y \geq 1,
\end{equation}
or equivalently
$$
\frac{h(y-1)}{h(y)}= \frac{g(0)}{g(1)} \frac{\mathbb{E}[1-q]}{\mathbb{E}[q]} y \frac{\mathbb{E}[p(1-p)^{y-1}]}{\mathbb{E}[(1-p)^{y}]},  \qquad \forall y \geq 1.
$$

Similarly, by taking  $y=1,  x\geq 1$, we get:
\begin{equation}
\label{part2}
B(1,1)G(1,x)= H(1,1)A(1,x) \Leftrightarrow \frac{A(1,x)}{G(1,x)}=\frac{B(1,1)}{H(1,1)}, \qquad \forall x \geq 1,
\end{equation}
or equivalently
$$
\frac{g(x-1)}{g(x)}= \frac{h(0)}{h(1)} \frac{\mathbb{E}[1-p]}{\mathbb{E}[p]} x \frac{\mathbb{E}[q(1-q)^{x-1}]}{\mathbb{E}[(1-q)^{x}]},  \qquad \forall x \geq 1.
$$
It is straightforward to check that $\frac{G(1,1)}{A(1,1)}\frac{B(1,1)}{H(1,1)} =1
$ by plugging $y=1$ in \eqref{part1} or \eqref{part2}.  Thus, we get:
$$
\frac{A(1,x)}{G(1,x)}=\frac{B(1,y)}{H(1,y)}, \qquad \forall x, y.
$$
This means that
\begin{equation}
	\label{backout}
\frac{\binom{x}{i}\mathbb{E}[q(1-q)^{x-1}]}{\mathbb{E}[(1-q)^{x}]}\frac{g(x-1)}{g(x)}=\frac{A(1,x)}{G(1,x)}=C=\frac{B(1,y)}{H(1,y)}=\frac{\binom{y}{i}\mathbb{E}[p(1-p)^{y-1}]}{\mathbb{E}[(1-p)^{y}]}\frac{h(y-1)}{h(y)}, \qquad \forall x, y,
\end{equation}
where $C$ is a constant.

Taking $x=2, y \geq 2$ in eq.\eqref{reducedform}, we get:
$$
A(0,2)H(0, y)+ A(1,2)H(1,y)+ A(2,2)H(2,y)= G(0,2)B(0,y)+G(1,2)B(1,y) +G(2,2)B(2,y).
$$
But the first terms on the LHS and RHS are equal, and the second terms on the LHS and RHS are also equal. Thus, we get:
$$
 A(2,2)H(2,y)=G(2,2)B(2,y).
$$
Hence,
$$
\frac{\mathbb{E}[p^2(1-p)^{y-2}]\mathbb{E}[(1-p)^{y-1}]}{\mathbb{E}[p(1-p)^{y-2}]\mathbb{E}[p(1-p)^{y-1}]}
$$
does not depend on $y.$ By \cite[page 3887]{gourieroux2019noncausal}, we deduce that $p$ follows the beta distribution. Therefore, using \eqref{backout}, we can deduce the expression of $h$ and check that it is a beta-NB pmf. Similarly, $q$ follows the beta distribution.

\subsection{Proof of Lemma \ref{ex6}}
\label{proofex6}
\paragraph{Proof for model \eqref{w1}-\eqref{w3}}


 By taking $Z=0$ in \eqref{pm1}, we have:
$
\mathbb{E}[X | Y, Z=0] = a_{11}+ a_{12} Y .
$
By taking expectation with respect to $Z=0$, we get:
$
\mathbb{E}[X | Z=0] = a_{11}+ a_{12} \mathbb{E}[Y|Z=0].
$
It is easily checked that $\mathbb{E}[Y|Z=0] =\lambda_2+\lambda_{12}$ and $\mathbb{E}[X | Z=0]=\lambda_1+\lambda_{12}$.

Thus, we get 
$
\lambda_1+\lambda_{12}= a_{11}+ a_{12} (\lambda_{2}+\lambda_{12}).
$
Similarly, by taking $Y=0$ in \eqref{pm1}, we also get
$
\lambda_1+\lambda_{13}= a_{11}+ a_{13} (\lambda_{3}+\lambda_{13}).
$
Finally, since we have $\mathbb{E}[X|Y=0, Z=0]=\lambda_1$, by taking both $Y=Z=0$ in \eqref{pm1}, we get
$
\lambda_1= a_{11}.
$

Let us now take the marginal expectation in \eqref{pm1}. We get:
$$
\lambda_1+\lambda_{12}+\lambda_{13}= a_{11}+ a_{12} (\lambda_{2}+\lambda_{12}+ \lambda_{23})+ a_{13} (\lambda_{3}+\lambda_{13}+\lambda_{23}).
$$
Comparing these three equations leads to $a_{12}=a_{13}=0$ and $a_{11}=\lambda_1+\lambda_{12}+\lambda_{13}=\mathbb{E}[X| Y=Z]$ in \eqref{pm1}.  This contradicts with $\mathbb{E}[X| Y=Z=0]=\lambda_1$. Thus, model \eqref{w1}-\eqref{w3} does not lead to linear conditional expectations.

\paragraph{Proof for model \eqref{w123}-\eqref{w123ter}}  We  take $Z=0$ in \eqref{pm1} and get
$
\mathbb{E}[X \mid Y, Z=0] = a_{11}+ a_{12} Y;
$
and by taking expectation with respect to $Z=0$, we get
$
\mathbb{E}[X \mid Z=0] = a_{11}+ a_{12} \mathbb{E}[Y\mid Z=0].
$
Then, we can check that $\mathbb{E}[Y\mid Z=0] =\lambda_2$ and $\mathbb{E}[X \mid Z=0]=\lambda_1$.  Thus, we get
$
\lambda_1=a_{11}+a_{12} \lambda_2.
$
Similarly, by taking $Y=0$ in \eqref{pm2}, we also get
$
\lambda_1=a_{11}+a_{13} \lambda_3.
$
Finally, since we have $\mathbb{E}[X\mid Y=0, Z=0]=\lambda_1$, by taking both $Y=Z=0$ in \eqref{pm1}, we get
$
\lambda_1= a_{11}.
$
This leads to $a_{12}=a_{13}=0$.  Hence, by taking marginal expectation, we get
$
a_{11}=\lambda_{1}+\lambda_{123}.
$
Therefore, we have arrived at a contradiction.

\subsection{Proof of Lemma \ref{marcheplusdimension3}}
\label{proofmarcheplusdimension3}

By the Bayes' formula, the conditional pmf of $N$ given $X$ and $Y$ is proportional to
\begin{align*}
\frac{\Gamma(n+\delta)}{\Gamma(\delta) n!} p_0^\delta (1-p_0)^n p_0^\delta  \frac{\Gamma(n+y+\delta)}{\Gamma(\delta+y) n!} p_1^{\delta+n} (1-p_1)^n  \frac{\Gamma(n+x+\delta)}{\Gamma(\delta+x) n!} p_2^{\delta+n} (1-p_2)^n.
\end{align*}
If we take $x=y=0$,  then this term is the pmf of a NB distribution with number of successes parameter $\delta$ and probability $1-(1-p_0)p_1 p_2$.  Thus,
$$
\mathbb{E}[N \mid X=Y=0]= \delta \cdot\frac{(1-p_0)p_1 p_2}{1-(1-p_0)p_1 p_2}.
$$
Similarly,  by taking $X=0$,  we get:
$$
\mathbb{E}[N \mid X=0, Y=y]= (\delta +y) \frac{(1-p_0)p_1 p_2}{1-(1-p_0)p_1 p_2}, \qquad \mathbb{E}[N \mid X=x, Y=0]= (\delta +x) \frac{(1-p_0)p_1 p_2}{1-(1-p_0)p_1 p_2}.
$$
Hence,  if $\mathbb{E}[N \mid X, Y]$  were linear, we ought to have:
$$
\mathbb{E}[N \mid X, Y]= (\delta +x+y) \frac{(1-p_0)p_1 p_2}{1-(1-p_0)p_1 p_2}.
$$
By taking marginal expectation, we get:
$$
\mathbb{E}[N  ]= (\delta +\mathbb{E}[X]+\mathbb{E}[Y]) \frac{(1-p_0)p_1 p_2}{1-(1-p_0)p_1 p_2}.
$$
Then, we plug in
$$
\mathbb{E}[N]=\delta \frac{1-p_0}{p_0}, \qquad \mathbb{E}[X]= \frac{1-p_1}{p_1}(\delta+ \mathbb{E}[N]), \qquad \mathbb{E}[Y]= \frac{1-p_2}{p_2}(\delta+ \mathbb{E}[N]),
$$
and after tedious but elementary algebra,  we get $(1-p_1)(1-p_2)=0$, which is a contradiction.  Therefore, we deduce that  $\mathbb{E}[N \mid X, Y]$ is not linear.

 \subsection{Proof of Lemma \ref{principalminor2}}
 \label{prooflemma1}
Let  $(X,Y)$ be a solution of (\ref{linearconditionalexp}) with positive variances. By  (\ref{linearconditionalexp}),  we have $Y=aX+b+\epsilon_1$,  where the residual of the regression $\epsilon_1$ is such that $\mathbb{E}[\epsilon_1\mid X]=0$.
Then, $
{\rm Cov}(Y, X)={\rm Cov}(aX+b+\epsilon_1, X)=a \mathbb{V}[X].
$
Similarly,  we have ${\rm Cov}(Y, X)=c \mathbb{V}[Y]$.  Thus,
$$
{\rm Cov}(Y, X)=a  \mathbb{V}[X]=c\mathbb{V}[Y].
$$
Hence
$
a=0\Leftrightarrow c=0,
$
and
$
ac=[{\rm Corr}(X,Y)]^2\in[0,1].
$

Finally,  if $ac=1$,  then ${\rm Corr}(X,Y)=\pm1$, i.e., there is a perfect linear relationship between $X$ and $Y$. Therefore,  (iii) holds by (\ref{linearconditionalexp}).

\subsection{Proof of Theorem \ref{abcde}}
\label{prooflemma3}
In this model, comparing \eqref{conditionalexpectation} with \eqref{linearconditionalexp}, we get:
$$
c=\theta_1-\theta_2,\ \ a=\theta_3-\theta_4,\ \ d=\delta\theta_1,\ \ b=\delta\theta_3.
$$
Then,
$$
\theta_1=\frac{d}{\delta},\ \ \theta_2=\frac{d}{\delta}-c,\ \  \theta_3=\frac{b}{\delta},\ \ \theta_4=\frac{b}{\delta}-a.
$$
We have:
$$
\theta_2 > -1\Leftrightarrow \frac{d}{\delta}+1>c;\ \ \ \ \theta_4 > -1\Leftrightarrow \frac{b}{\delta}+1>a,
$$
and
\begin{eqnarray*}
(\ref{sameproduct})\ {\rm holds}
&\Leftrightarrow&
\left(\frac{b}{\delta}-a\right)\left(\frac{d}{\delta}+\frac{b}{\delta}\cdot c\right)=\left(\frac{d}{\delta}-c\right)\left(\frac{b}{\delta}+\frac{d}{\delta}\cdot a\right)\nonumber\\
&\Leftrightarrow&(b-a\delta)(d+bc)=(d-c\delta)(b+ad)\nonumber\\
&\Leftrightarrow&[ad(1-c)+bc(a-1)]\delta=b^2c-ad^2.
\end{eqnarray*}

\noindent (i) Suppose that $a\ge1,ac<1$, i.e., $a>1,ac<1$ or $a=1,c<1$. Then,
$$
\delta>0\Leftrightarrow \left(\frac{b}{d}\right)^2>\frac{a}{c}.
$$
Note that if
$$
\left(\frac{b}{d}\right)^2>\frac{a}{c},
$$
then
\begin{eqnarray*}
\frac{d}{\delta}+1-c&=&\frac{ad^2(1-c)+bcd(a-1)+(1-c)(b^2c-ad^2)}{b^2c-ad^2}\\
&=&\frac{b(1-c)+d(a-1)}{b^2c-ad^2}\cdot bc>0,
\end{eqnarray*}
and
\begin{eqnarray*}
\frac{b}{\delta}+1-a&=&\frac{abd(1-c)+b^2c(a-1)+(1-a)(b^2c-ad^2)}{b^2c-ad^2}\\
&=&\frac{b(1-c)+d(a-1)}{b^2c-ad^2}\cdot ad>0.
\end{eqnarray*}

\noindent (ii) Suppose that $c\ge1,ac<1$, i.e., $c>1,ac<1$ or $c=1,a<1$. Then,
$$
\delta>0\Leftrightarrow \left(\frac{b}{d}\right)^2<\frac{a}{c}.
$$
Note that  if
$$
\left(\frac{b}{d}\right)^2<\frac{a}{c},
$$
then
$$
\frac{d}{\delta}+1-c>0,\ \ \ \ \frac{b}{\delta}+1-a>0.
$$

\noindent (iii) Suppose that $a,c<1$ and $ad(1-c)+bc(a-1)=0$, which is equivalent to
$$
\frac{a}{c}=\frac{b+ad}{d+bc}.
$$
Then,
$$
\delta>0\Leftrightarrow\left(\frac{b}{d}\right)^2=\frac{a}{c}.
$$
Plugging $\frac{b}{d}=\sqrt{\frac{a}{c}}$ into $
\frac{a}{c}=\frac{b+ad}{d+bc},
$
we get $a=c$,  thus $b=d$.

\noindent (iv) Suppose that $a,c<1$ and
$$
\frac{a}{c}\not=\frac{b+ad}{d+bc}.
$$
Then,
$$
\delta>0\Leftrightarrow \left(\frac{b}{d}\right)^2>\frac{a}{c}>\frac{b+ad}{d+bc}, \ {\rm or}\ \left(\frac{b}{d}\right)^2<\frac{a}{c}<\frac{b+ad}{d+bc}.
$$
We can easily check that the first chain of inequality is only possible when $a>c$ whereas the second is only possible when $a<c$.

\subsection{Proof of Theorem \ref{farkas}}
Suppose that $0<a,c<1$. We choose a perfect square number $N$ (i.e.,  $\sqrt{N}$ is an integer) such that $N>4$ and
$$
\frac{1}{\sqrt{N}}<a,c<1-\frac{1}{\sqrt{N}};\ \ \ \  b,d<\sqrt{N}.
$$
Then,
$$
b<(1-a)N,\ \ \ \ d<(1-c)N.
$$

By Farkas' lemma, the linear system of equations \eqref{farka1}-\eqref{farka3} has non-negative solutions $\{p_{00},\dots,p_{0N},p_{10}$, $\dots,p_{1N},\dots,p_{NN}\}$ if and only if there does not exist $y_0<0$ and $y_1,\dots,y_{2N+2}\in\mathbb{R}$ such that all the following $(N+1)^2$ inequalities hold:
\begin{eqnarray*}
y_0+(ai-j+b)y_{i+1}+(-i+cj+d)y_{j+N+2}\ge 0,\ \ \ \ j=0,\dots,N;i=0,\dots,N,
\end{eqnarray*}
equivalently,
\begin{eqnarray}\label{ij}
(ai-j+b)y_{i+1}+(-i+cj+d)y_{j+N+2}>0,\ \ \ \ j=0,\dots,N;i=0,\dots,N.
\end{eqnarray}

We will show that \eqref{linearconditionalexp} has non-negative solutions by showing that there does not exist $y_0<0$ and $y_1,\dots,y_{2N+2}\in\mathbb{R}$ such that (\ref{ij}) holds. Suppose on the contrary that there exist $y_0<0$ and $y_1,\dots,y_{2N+2}\in\mathbb{R}$ such that (\ref{ij}) holds. Note that (\ref{ij}) with $i=j=0$ becomes
$$
by_{1}+dy_{N+2}>0.
$$
Hence, we deduce that $y_1>0$ or $y_{N+2}>0$.

\noindent (i) Suppose that $y_1>0$. Note that (\ref{ij}) with $i=0,j=\sqrt{N}$ becomes
$$
(-\sqrt{N}+b)y_{1}+(c\sqrt{N}+d)y_{N+\sqrt{N}+2}>0.
$$
Then, $y_{N+\sqrt{N}+2}>0$. Note that (\ref{ij}) with $i=0,j=N$ becomes
$$
(-N+b)y_{1}+(cN+d)y_{2N+2}>0.
$$
Then, $y_{2N+2}>0$.  Note that (\ref{ij}) with $i=N,j=N$ becomes
$$
(aN-N+b)y_{N+1}+(-N+cN+d)y_{2N+2}>0,
$$
which implies that $y_{N+1}<0$. Note that (\ref{ij}) with $i=N,j=\sqrt{N}$ becomes
$$
(aN-\sqrt{N}+b)y_{N+1}+(-N+c\sqrt{N}+d)y_{N+\sqrt{N}+2}>0,
$$
which implies that $y_{N+1}>0$. We have arrived at a contradiction.

\noindent (ii) Suppose that $y_{N+2}>0$. Note that (\ref{ij}) with $i=\sqrt{N},j=0$ becomes
$$
(a\sqrt{N}+b)y_{\sqrt{N}+1}+(-\sqrt{N}+d)y_{N+2}>0.
$$
Then, $y_{\sqrt{N}+1}>0$. Note that (\ref{ij}) with $i=N,j=0$ becomes
$$
(aN+b)y_{N+1}+(-N+d)y_{N+2}>0.
$$
Then, $y_{N+1}>0$.  Note that (\ref{ij}) with $i=N,j=N$ becomes
$$
(aN-N+b)y_{N+1}+(-N+cN+d)y_{2N+2}>0,
$$
which implies that $y_{2N+2}<0$. Note that (\ref{ij}) with $i=\sqrt{N},j=N$ becomes
$$
(a\sqrt{N}-N+b)y_{\sqrt{N}+1}+(-\sqrt{N}+cN+d)y_{2N+2}>0.
$$
We have arrived at a contradiction since
$$
a\sqrt{N}-N+b<-(1-a)N+b<0,\ \ y_{\sqrt{N}+1}>0,\ \ -\sqrt{N}+cN+d>d,\ \ y_{2N+2}<0.
$$

Hence,  there does not exist $y_0<0$ and $y_1,\dots,y_{2N+2}\in\mathbb{R}$ such that (\ref{ij}) holds. Therefore, by Farkas' lemma, we deduce that \eqref{linearconditionalexp} has a solution.

\subsection{Proof of Theorem \ref{principalminor}}
\label{prooflemma4}
Let us focus on the proof of the case where $n=3$.  The case with arbitrary dimension is completely similar.

By \eqref{pm1}, we get $X=a_{11}+a_{12}Y+a_{13}Z+\epsilon_1$, where $\mathbb{E}[\epsilon_1 \mid Y, Z]=0$.  Let us compute ${\rm Cov}(X,  \epsilon_1)$.  On the one hand, we have
$$
{\rm Cov}(X,  \epsilon_1)={\rm Cov}(a_{11}+a_{12}Y+a_{13}Z+\epsilon_1, \epsilon_1)= \mathbb{V}[\epsilon_1].
$$
On the other hand, we have
$$
{\rm Cov}(X,  \epsilon_1)= {\rm Cov}(X,  X-a_{11}-a_{12}Y-a_{13}Z)=\mathbb{V}[X]-a_{12}{\rm Cov}(X, Y)-a_{13}{\rm Cov}(X, Z).
$$
Thus, $$
 \mathbb{V}[\epsilon_1]=\mathbb{V}[X]-a_{12}{\rm Cov}(X, Y)-a_{13}{\rm Cov}(X, Z).
$$
Similarly,  we have that
\begin{align*}
 \mathbb{V}[\epsilon_2]&=\mathbb{V}[Y]-a_{21}{\rm Cov}(X, Y)-a_{23}{\rm Cov}(Y, Z),\\
  \mathbb{V}[\epsilon_3]&=\mathbb{V}[Z]-a_{31}{\rm Cov}(X,Z)-a_{32}{\rm Cov}(Y, Z).
\end{align*}
Hence, in matrix form,
$$
A:=\begin{bmatrix}
1 & -a_{12}  & -a_{13} \\
-a_{21} & 1  & -a_{23} \\
-a_{31} & -a_{32}  & 1
\end{bmatrix}=
DC^{-1},
$$
where $$
C=\begin{bmatrix}
\mathbb{V}[X]& {\rm Cov}(X, Y)  & {\rm Cov}(X, Z) \\
{\rm Cov}(X, Y) & \mathbb{V}[Y] & {\rm Cov}(Y, Z) \\
{\rm Cov}(X, Z) & {\rm Cov}(Y, Z)  & \mathbb{V}[Z]
\end{bmatrix}
$$ 
is the covariance matrix of $(X, Y, Z)$ and $D$ is the diagonal matrix with diagonal elements $\mathbb{V}[\epsilon_1],  \mathbb{V}[\epsilon_2], \mathbb{V}[\epsilon_3]$.  

Since $X, Y, Z$ are assumed linearly independent,  all the diagonal elements of $D$ are positive and $C$ is symmetric positive definite. Thus $\det A= \det  (DC^{-1 } )= \det( C^{-1/2} D   C^{-1/2}  )>0$ since $C^{-1/2} D   C^{-1/2}$ is also symmetric positive definite.

As for the principal minors of $A$ of size 2,  let us, without loss of generality,  focus on $\det A_{\{1,2\}}$ corresponding to the determinant of the first two rows and columns of $A$.  Since $D$ is diagonal,  by block matrix calculation,  we have:
$$
A_{\{1,2\}}= [C^{-1}]_{\{1,2\}} D_{\{1,2\}},
$$
 where $[C^{-1}]_{\{1,2\}} $, obtained as the first two rows and columns of $C^{-1}$, is symmetric positive definite and $
D_{\{1,2\}}$ is diagonal.  Thus,  we deduce that $\det A_{\{1,2\}}=1-a_{12}a_{21}>0$.  Similarly,  we also have that $\det A_{\{1,3\}}>0$,  $\det A_{\{2,3\}}>0$. Hence, all the principal minors of matrix $A$ are positive.

By  \eqref{pm1}-\eqref{pm3} and the tower property of conditional expectation, we get:
\begin{eqnarray*}
&&\mathbb{E}[X\mid Y]=a_{11}+a_{12}Y+a_{13}\mathbb{E}[Z\mid Y],\label{r211}\\
&&\mathbb{E}[Z\mid Y]=a_{31}\mathbb{E}[X\mid Y]+a_{32}Y+a_{33}.\label{r266}
\end{eqnarray*}
Since the principal minor $A_{\{1,2\}}=1-a_{12}a_{21}$ is positive,  we can invert the system and express $\mathbb{E}[X\mid Y]$ and $\mathbb{E}[Z\mid Y]$ as affine functions of $Y$.  Similarly,  the other pairwise conditional expectations, such as $\mathbb{E}[Y\mid X]$,  are also all linear in their respective conditioning variables.



\end{document}